\RequirePackage{fix-cm}
\documentclass[smallextended]{svjour3}       
\smartqed  
\usepackage{graphicx}
\usepackage{subcaption}
\usepackage[english]{babel}
\usepackage{appendix}
\usepackage{enumitem}
\usepackage{url}            
\usepackage{booktabs}       
\usepackage{amsfonts}       
\usepackage{nicefrac}       
\usepackage{microtype}      
\usepackage{lipsum}
\usepackage{amssymb}
\usepackage{amsmath}
\usepackage{mathptmx}
\usepackage{hyperref}       
\usepackage{xcolor}
\usepackage{comment}
\usepackage{soul}
\usepackage{comment}

\usepackage{algorithm}
\usepackage[noend]{algpseudocode}
\makeatletter
\def\BState{\State\hskip-\ALG@thistlm}
\makeatother

\usepackage[english]{babel}
\DeclareMathOperator*{\esssup}{ess\,sup}

\DeclareMathOperator*{\argmin}{arg\,min}

\newcommand{\dd}{\mathrm{d}}
\newcommand{\DD}{\mathrm{D}}

\newcommand{\1}{\mathbf{1}}

\newcommand{\ie}{\emph{i.e.}}

\begin{document}

\title{Stackelberg Risk Preference Design*
}


\author{Shutian Liu         \and
        Quanyan Zhu 
}


\institute{
*This version: December, 2022.
\\
Corresponding Author: S. Liu \at
Affiliation: Department of Electrical and Computer Engineering,               Tandon School of Engineering, New York University \\
              \email{sl6803@nyu.edu}           
           \and
           Q. Zhu \at
Affiliation: Department of Electrical and Computer Engineering,               Tandon School of Engineering, New York University \\
              \email{qz494@nyu.edu} 
}

\date{Received: date / Accepted: date}

\maketitle

\begin{abstract}
Risk measures are commonly used to capture the risk preferences of decision-makers (DMs). The decisions of DMs can be nudged or manipulated when their risk preferences are influenced by factors such as the availability of information about the uncertainties. This work proposes a Stackelberg risk preference design (STRIPE) problem to capture a designer's incentive to influence DMs' risk preferences. STRIPE consists of two levels. In the lower level, individual DMs in a population, known as the followers, respond to uncertainties according to their risk preference types. In the upper level, the leader influences the distribution of the types to induce targeted decisions and steers the follower's preferences to it. Our analysis centers around the solution concept of approximate Stackelberg equilibrium that yields suboptimal behaviors of the players. We show the existence of the approximate Stackelberg equilibrium. The primitive risk perception gap, defined as the Wasserstein distance between the original and the target type distributions, is important in estimating the optimal design cost. We connect the leader's optimality compromise on the cost with her ambiguity tolerance on the follower's approximate solutions leveraging Lipschitzian properties of the lower level solution mapping. To obtain the Stackelberg equilibrium, we reformulate STRIPE into a single-level optimization problem using the spectral representations of law-invariant coherent risk measures. We create a data-driven approach for computation and study its performance guarantees. We apply STRIPE to contract design problems under approximate incentive compatibility. Moreover, we connect STRIPE with meta-learning problems and derive adaptation performance estimates of the meta-parameters.




\subclass{MSC code1 \and MSC code2 \and more}
\end{abstract}

\section{Introduction}
\label{sec:intro}
Risk preference describes the perception of losses or gains when DMs face random outcomes.
Modeling of risk dates back to the utility theory of von Neumann and Morgenstern \cite{von2007theory}, where decisions of rational players seek to maximize their expected utility or satisfaction. The expected utility leads to decisions that are often referred to as risk-neutral.
However, human risk perception often exhibits nonlinearity \cite{kahneman2013prospect}.  
Risk-aversion is a phenomenon where DMs tend to prefer outcomes with low uncertainties to high ones, even when the expected utility of the former outcome is lower. 
The celebrated Arrow-Pratt measure of relative risk aversion \cite{meyer2005relative} utilizes the nonlinearity of the utility functions to capture the risk preferences of DMs.
The cumulative prospect theory of Tversky and Kahneman \cite{tversky1992advances} enriches the description of risk attitudes by distinguishing losses from gains and by incorporating probability weightings on the cumulative distribution functions, as in the proposed rank-dependent utility model.
Other related approaches include the dual utility theory, stochastic orders and etc., we refer the readers to \cite{yaari1987dual,wang1998ordering,levy1992stochastic,dentcheva2003optimization}
and the references therein.

The seminal paper by Aztner et al. \cite{artzner1999coherent} has introduced coherent risk measures (CRMs) to quantify risk. The coherency holds if a risk measure satisfies four axioms, namely, monotonicity, convexity, translation equivariance, and positive homogeneity. There is a vast literature on risk quantification using the axiomatic approach, see, for example, the monographs \cite{follmer2016stochastic,pflug2014multistage} and the references therein.
One of the most essential properties of a CRM is its dual representation \cite{artzner1999coherent,follmer2016stochastic,ruszczynski2006optimization}. 
The representation conveniently transforms the computation of risk into a problem of finding the worse-case probability density function associated with the random cost function. Another line of research called distributionally robust optimization \cite{wiesemann2014distributionally} is closely related to CRMs due to this dual representation \cite{shapiro2017distributionally}.
A quintessential example of CRMs is the average value-at-risk (AV@R), which is a cumulative version of the value-at-risk (V@R), or the left side quantile. V@R has been used as the standard measure of risk for various applications in finance until the introduction of the AV@R due to the non-coherency of V@R.
The significance of AV@R in CRMs has been made more pronounced by their Kusuoka representations \cite{kusuoka2001law}.

The literature on risk measures focuses on the analysis or applications of given risk measures. 
However, individual risk preferences can be affected by exogenous manipulations or influences. 
Empirical evidence of the instability of human risk preferences has been observed in various scenarios.
In \cite{hanaoka2018risk}, the authors have discovered that there is a positive correlation between the frequency of gambling and an earthquake. This phenomenon has shown that the experience of natural disasters increases the level of risk tolerance of individuals.
In \cite{schildberg2018risk}, the author has leveraged the notions of preference stability from microeconomics and personality traits from personality psychology and has shown that, while the mean-level of individual risk preferences changes abruptly according to exogenous shocks, such as economic crises and violent conflicts, there is an overall tendency of becoming more risk-averse as one grows older.
The works \cite{levin2007stability,dohmen2016time,barseghyan2011risk} have also recorded empirical evidence of the instability of risk preferences.
The psychological evidence of the instability of risk perceptions has also been documented.
In \cite{slovic2006risk,slovic1995construction}, a psychological theory has been adopted to study the change or evolution of preference when DMs are subject to exogenous influence. The findings indicate that individual preferences can be manipulated through nudging, marketing, or propaganda.
More recently, in \cite{yuen2020psychological}, the authors have found that 
risk preferences exhibit population-level patterns due to social influences.

In this paper, we formally model the risk influence problem by proposing a Stackelberg risk preference design (STRIPE) framework to enable the design of population-level risk preferences. In particular, we introduce risk preference types (RPTs) to a population of agents making decisions under uncertainties. The risk attitude of individuals with an RPT is represented by a risk measure. The follower in the Stackelberg game is a population of indistinguishable agents, which can be equivalently represented by an idiosyncratic individual from the population whose average utility captures the one for the population. The leader in the game is a designer who determines the target distribution of RPTs and steers the population to it to achieve the leader’s design objective. Influencing the risk attitude of the population is nontrivial and costly. The effort to maneuver the RPT is quantified by the deviation of the target distribution from the uncontrolled RPT distribution.

One feature of our STRIPE framework is to consider the behavior of the population. Individuals of the same RPT in the population are viewed statistically indistinguishable.  Hence, the RPT of the population is represented by the RPT distribution. The distributional perspective can empower the interpretation of a mixed strategy of the leader where the leader and the follower interact repeatedly for a sufficiently long period of time. The STRIPE problem has a bilevel structure, allowing a customizable design for a population with specific characteristics and application-driven constraints.
At the lower level, an idiosyncratic agent who bears the averaged RPT is used to represent the population. This agent's decision under the uncertainty faced by the population is considered as the follower's action in the Stackelberg game.
At the upper level, the designer chooses an RPT distribution that jointly optimizes the design target as a function of the follower's action and the cost of RPT distribution manipulation.

The analysis of the STRIPE framework centers around the solution concept of approximate Stackelberg equilibrium that yields an outcome that corresponds to the suboptimal behaviors of both the leader and the follower.
The stochasticity of the follower's problem induces the suboptimality of the follower's action.
This suboptimality allows a set of $\epsilon$-solutions that are anticipated by the leader, naturally leading to approximate Stackelberg equilibrium solutions.

We define the primitive risk perception gap, which measures the distance between the uncontrolled RPT distribution and the target distribution that solicits the desired follower’s behavior. This gap is then used to characterize approximate solutions to the leader’s problem given candidate RPT distributions.
In particular, we derive an upper bound on the leader's cost represented by the deviation of the $\epsilon$-optimal sets given the uncontrolled RPT distribution and the target distribution using the primitive risk perception gap. 
The results build on the growth conditions for the stability analysis of optimization problems.

The suboptimal behaviors of the leader and the follower are interdependent.
Leveraging mathematical tools from set-valued analysis, we show that, under continuity and differentiability assumptions on the cost functions,
the leader's design of RPT distribution can be calibrated by the difference between a follower's response and his anticipated follower solution.
Furthermore, we connect the leader's optimality compromise of her objective with her ambiguity tolerance of the follower's approximate solutions. 
In particular, the leader's optimality compromise induced by her ambiguity tolerance is upper bounded logarithmically.

To develop analytical and computational solutions to the STRIPE framework, we transform the follower’s risk-sensitive stochastic optimization problems in terms of AV@Rs with the aid of the Kusuoka representation theorem. Then, by approximating the risk spectrum using step functions, we reformulate the follower’s problem into a convex optimization problem when the cost functions are convex. Finally, the STRIPE problem is cast into a tractable single-level optimization problem using the optimal value reformulation. 
A data-driven approach to dealing with stochasticity based on the sample average approximation (SAA) technique for stochastic programs is introduced.
We extend the classic finite sample performance guarantee of SAA and show that $\epsilon$-optimal follower's solution can be obtained with high probability given sufficiently many samples.

Contract design is a class of Principal-Agent (P-A) problems that boast the feature of bi-level structures. The leader is the contract designer or the principal who optimizes his own utility by offering to the agent or the follower a contractual agreement on the resource flow between them.
The follower observes the contract and decides whether to participate. 
An incentive-compatible contract encourages the agent to take action on behalf of the principal and guarantees beneficial participation.
The STRIPE framework enables the joint design of risk preference and the contract. 
The participant is an idiosyncratic individual of a population of agents with various RPTs.
The contract designer is provided with an additional degree of freedom, the RPT distribution of the population, when she determines the contract. 
We use a contract design problem under approximate incentive compatibility as a case study.
We show that, when the principal holds a pessimistic perspective, the consideration of $\epsilon$-approximate incentive compatibility induces an optimality gap and decreases the principal's revenue from the one obtained using exact incentive compatibility at the worst possible rate of $O(\epsilon^{1/2})$.
The contract problem can be solved using the sampled reformulation introduced in Section \ref{sec:reformulation}.

Another key application of STRIPE is in machine learning and data science. We connect our framework with the adversarial meta-learning problem, where the leader chooses the learning task distribution, and the follower performs meta-learning. The RPT represents a learning task, and the follower determines a meta-parameter that optimizes the average performance of a set of learning tasks. The risk-sensitivity of the follower’s problem in our framework leads to a risk-sensitive adversarial learning task, which is closely related to a distributionally robust machine learning problem.
The leader’s objective represents exogenous guidance on the meta-parameter. This guidance can incorporate past experiences, expert advice, and mean-level standards on the meta-parameter, enriching the calibration of the meta-parameter from only using the test data.
Leveraging the sensitivity results of the optimal value of the follower's problem, we show that the complete performance estimate of the adaptations of a meta-parameter to the learning tasks of interest can be obtained.
The performance estimates can be utilized to either improve the training of the meta-parameter or reduce the computation requirement in the lower level problem.

We briefly review the related literature in the next section. Section \ref{sec:problem} presents the formulation of the STRIPE framework, the solution concept, and the equilibrium existence results.
In Section \ref{sec:analysis}, we first provide the estimation errors of candidate solutions; then, we elaborate on parameter selection in the approximate Stackelberg equilibrium.
We discuss analytical and computational solutions to the STRIPE problem in Section \ref{sec:reformulation}.
Selected applications of STRIPE are presented in Section \ref{sec:case}.
Section \ref{sec:conclusion} concludes the paper.

\section{Related Works}
\label{sec:related}
CRMs are useful to capture different subjective risk attitudes \cite{pichler2017quantitative}.
There is a recent growing body of literature that studies the selection or design of risk preferences. 
The spectral risk measures introduced in \cite{acerbi2002spectral} provide a convenient way of representing a CRM using its associated risk spectrum. The risk spectrum weights the V@R of the random loss for different probability parameters. 
The advantage of spectral risk measures in the design of risk preference arises from the monotonicity of the risk spectrum and its convenience for approximation.
In \cite{acerbi2002portfolio}, the authors have proposed a  method of computation by leveraging this property and the equivalence of an AV@R to a convex optimization problem \cite{rockafellar2000optimization}.
This methodology and the corresponding analysis of spectral risk measures are extended in \cite{shapiro2013kusuoka}.
The computational method of the STRIPE problem also builds on spectral representation. 
Furthermore, since we are interested in a single-level reformulation of the Stackelberg game, the tractability and the constraint qualification issues of the single-level reformulation rely heavily on the computation of the risks of the follower's problem. 
In \cite{guo2021robust}, the authors have put forward a robust spectral risk optimization framework where the choice of risk spectrum aims to be robust to an ambiguity set defined by a group of moment-type constraints. These constraints have the interpretation of elicited preference information, making the selected risk spectra capture individual risk preferences. 
A sorting-free computational method based on the approximation of risk spectra has been discussed in \cite{guo2021robust}.
The work has also provided an ample amount of spectrum and probability  approximation errors, extending the quantitative stability results discussed in \cite{romisch2003stability,pichler2018quantitative,shapiro1994quantitative,wang2020robust}. 

The selection of risk measures has also been investigated from an axiomatic perspective. 
In \cite{delage2018minimizing}, the authors have proposed optimizing financial positions subject to subsets of axiomatic properties of risk measures. 
Referred to as preference robust risk minimization, the approach in \cite{delage2018minimizing} constructs risk measures based on the most pessimistic assessment elicited from users' choice preference information.
The attempt to risk preference selection introduced in \cite{liu2020robust} focuses on balancing between the worst-case loss and expected loss.
More recently, in \cite{li2021inverse}, the author has considered an inverse optimization framework for selecting risk measures. 
Through inverse optimization, a risk measure is designed to satisfy predetermined axiomatic properties and optimizes, for example, the deviation from a reference risk measure.
The approach in \cite{li2021inverse} results in tractable convex programs and, unlike many standard inverse optimization problems, is nonparametric.
In the STRIPE framework, the objective function of the leader can be considered constructed from reference risk measures or elicited assessment standards.
As we will discuss in later sections, we assume that an anticipated follower's decision exists when we present the estimation of equilibrium solutions.
We take a game-theoretic approach and focus on the risk preference distribution in populations, which naturally leads to the mixed strategies of the leader.

Another line of research related to us is the literature on ambiguity set construction in distributionally robust optimization problems.
Since the ambiguity set of distributions captures individual risk attitudes towards distributional uncertainties, the construction of ambiguity sets is essentially a design of the risk attitude of a decision-maker. We refer the readers to \cite{bertsimas2009constructing,wiesemann2014distributionally,delage2010distributionally} for the construction techniques.

This work builds on and contributes to the literature on bilevel programming. Classic results on optimality conditions and constraint qualifications can be found in \cite{ye1997exact,dempe2002foundations,dempe2013bilevel}.
In \cite{lin2014solving}, the authors have focused on simple bilevel programming problems (Stackelberg games) where the lower level problem is nonconvex.
They have developed a smoothing projected gradient algorithm for solving the nonsmooth nonconvex single-level reformulation of the bilevel program.
In particular, an approximate bilevel program where an $\epsilon$-optimal lower level solution is considered.
This relaxation makes it easier for the nonsmooth Mangasarian-Fromovitz constraint qualification (MFCQ) to hold in the single-level reformulation.
Our work is related to the recent work \cite{burtscheidt2020risk}, where the authors have considered a risk-averse two-stage bilevel stochastic linear program.  
The stochasticity arises from the fact that the follower's action is taken at the second stage.
The authors have shown the fundamental properties of the leader's choice function and have reformulated the stochastic bilevel problems to standard bilevel problems for several popular risk measures under discrete distributions. 
In the STRIPE framework, we will also discuss approximated Stackelberg solution where the lower level problem is only assumed to be solved approximately. 
However, this consideration arises from the fact that the lower level problem is a stochastic programming problem.

The P-A problems consist of an important class of bilevel programs.
For the background on P-A problems and contract theory we refer to the monograph \cite{stole2001lectures} and the references therein.
In particular, the moral hazard issue caused by information asymmetry in P-A problems is known to decrease the performance of contracts when the agent is risk-averse \cite{chade2002risk}.
Recently, in \cite{liu2022mitigating}, we have taken the cyber insurance problem as an example of a class of P-A problems and investigated the role of risk preference design. 
We have proposed a metric to quantify moral hazard and have shown that with the additional degree of freedom of the principal granted by the RPT distribution, the risk design can mitigate moral hazard.
However, \cite{liu2022mitigating} is a framework for a specific application. We will discuss the connections between our STRIPE framework and P-A problems in Section \ref{sec:case}.
Bilevel programs have also been recognized as a useful tool for parameter selection, which is one of the central questions in machine learning, especially meta-learning \cite{vanschoren2018meta}.
In \cite{bennett2006model}, the authors have introduced bilevel optimization to cross-validate for hyper-parameter selection.
In \cite{ye2021difference}, the authors have designed an algorithm for solving bilevel programs with an emphasis on applications related to hyperparameter selection.
The proposed algorithm is particularly powerful when the lower level problem is convex as seen in problems related to support vector machines or least absolute shrinkage and selection operators. 
Our STRIPE framework suits applications related to meta-learning. 
As we will show in Section \ref{sec:case}, each RPT defines a risk-sensitive learning task, and the upper level objective function can be interpreted as a test criterion on the meta-parameter.

\section{Problem Formulation}
\label{sec:problem}
In this section, we first endow a population of decision-makers with what we refer to as RPTs.
Then, we propose the STRIPE problem, a Stackelberg game involving one leader and one follower, to enable the design of the risk preferences.
In this game, we consider the average response from the population  as the follower's action.
The leader finds the optimal distribution of risk preferences of the population to minimize her loss.
The solution concept that we propose extends the standard approximate Stackelberg solution and fits the setting where the follower responds to a random environment.

\subsection{Risk Preference Types}
\label{sec:problem:types}
Consider a population of mass $1$ of decision-makers.
This population faces uncertainties modeled by the probability space $(\Xi, \mathcal{F})$ with the reference probability measure denoted by $P$. 
We use $\xi\in\Xi$ to denote a sampled outcome.
Let $x\in X\subset \mathbb{R}^n$ denote the decision variable of an individual of interest from the population.
We use $Z:=f(x,\xi)$ to denote the random loss of this individual, where 
$Z:\Xi\rightarrow \mathbb{R}$ is a measurable function with the finite $p$-th order moment from the space $\mathcal{Z} :=\mathcal{L}_p(\Xi, \mathcal{F}, P)$. 
The parameter $p$ lives in $[1,+\infty)$.
Note that for notational simplicity, we will occasionally use $Z_x$ to represent $f(x,\xi)$ to emphasize its dependence on $x$.
In this paper, we assume that the loss function $f(\cdot, \xi)$ is convex for all $\xi\in \Xi$.
As we will discuss in later sections that this assumption results in a convex lower level problem.

Since we endow different risk preferences to different DMs in the population, they perceive distinct risks when interacting with the underlying stochasticity.
In particular, each individual from the population is identified with an RPT $\theta\in\Theta\subset\mathbb{R}$.
Each RPT labels a DM with a risk preference which influences the DM's quantification of costs under uncertainty.
In the population setting, the RPT can be interpreted as the nominal risk attitude of a group of individuals.
An RPT $\theta$ is represented by a risk measure $\rho_\theta:\mathcal{Z}\rightarrow \mathbb{R}$.

If not specified otherwise, we proceed with the assumption that for all $\theta\in\Theta$, the risk measure $\rho_\theta$ is a coherent risk measure \cite{artzner1999coherent}.
The definition of coherent risk measures is presented below for completeness.
\begin{definition}
(Coherent risk measures.) A function $\rho:\mathcal{Z}\rightarrow\mathbb{R}$ is referred to as a coherent risk measure if it satisfies the following axioms: \\
\hspace*{1em}(A1) Monotonicity: If $Z, Z'\in\mathcal{Z}$ and $Z\succeq Z'$, then $\rho(Z)\geq \rho(Z')$.\\
\hspace*{1em}(A2) Convexity: 
    $\rho(tZ+(1-t)Z') \leq t\rho(Z)+(1-t)\rho(Z')$
for all $Z, Z'\in\mathcal{Z}$ and $t\in[0,1]$.\\
\hspace*{1em}(A3) Translation equivariance: If $Z\in\mathcal{Z}$ and $a\in\mathbb{R}$, then $\rho(Z+a)=\rho(Z)+a$.\\
\hspace*{1em}(A4) Positive homogeneity: If $Z\in\mathcal{Z}$ and $t\in\mathbb{R}_+$, then $\rho(tZ)=t\rho(Z)$.
\end{definition}
The relation $Z\succeq Z'$ in axiom (A1) means $Z(\xi)\geq Z'(\xi)$ for a.e. $\xi\in\Xi$.

One of the most powerful consequences of coherent risk measures is the dual representation \cite{artzner1999coherent,ruszczynski2006optimization}.
Let $\mathcal{Z}^*:=\mathcal{L}_q(\Xi, \mathcal{F}, P)$ for $q\in (1,\infty]$ denote the dual space of $\mathcal{Z}$, \ie, $\frac{1}{p}+\frac{1}{q}=1$.
Following \cite{artzner1999coherent,ruszczynski2006optimization}, the dual representation the risk resulting from the RPT $\theta$ is
\begin{equation}
    \rho_\theta[f(x,\xi)]=\sup_{\nu \in \mathfrak{m}_\theta}
    \int_{\Xi}f(x,\xi)\nu(\xi)dP(\xi)=\sup_{\nu \in \mathfrak{m}_\theta} \langle\nu, f(x,\xi)\rangle,
    \label{eq:dual representation of individual risk}
\end{equation}
where $\mathfrak{m}_\theta \subset \mathcal{Z}^*$ denotes the dual set associated with the risk measure $\rho_\theta$.
The set $\mathfrak{m}_\theta$ is convex and compact when $p\in[1,+\infty)$ \cite{shapiro2021lectures}. 
Hence, the maximum of (\ref{eq:dual representation of individual risk}) is attainable.
Let $\nu^*_\theta$ denote the argument which attains the maximum given RPT $\theta$.
Then, the follower's cost given a measure $\mu\in\mathcal{Q}$ becomes:
\begin{equation}
        U_{\mu}(x)=\int_{\Theta}\int_{\Xi} f(x,\xi)\nu^*_\theta (\xi)dP(\xi)d\mu(\theta).
        \label{eq:dual representation of follower cost}
\end{equation}
For notational simplicity, we define $F(x,\theta):=\int_{\Xi} f(x,\xi)\nu^*_\theta (\xi)dP(\xi)$.

The following two examples show both cases where the space of RPTs $\Theta$ is a continuum, and $\Theta$ has only finitely many elements.
\begin{example}
\label{eg:1}
(Finite RPT space.) Let $\Theta=\{1,2,3\}$. The risk preferences are $\rho_1(Z)=\mathbb{E}[Z]$, $\rho_2=\esssup [Z]$, and $\rho_3=\sqrt{\mathbb{E}[Z^2]-(\mathbb{E}[Z])^2}$, respectively.
\end{example}
\begin{example}
\label{eg:2}
(RPT space containing parameterized risk measures.) Let $\Theta=(0,1)$. A risk preference of type $\theta\in (0,1)$ is represented by $\rho_\theta(Z)=\text{AV@R}_\theta (Z)$.
\end{example}

Note that a finite type space is commonly used for modeling the population-wise risk preferences. 
Individuals can be classified into subpopulations based on how close their behaviors or decision habits are.
Then, nominal risk preferences can be obtained for each subpopulation through elicitation.

\subsection{Stackelberg Risk Preference Design Problem}
\label{sec:problem:stackelberg}
To model the follower in the STRIPE problem, we consider a distribution over the RPT space $\Theta$.
Specifically, let $(\Theta, \mathcal{G})$ denote the probability space of RPTs, with $\mathcal{Q}$ denoting the set of probability measures on $(\Theta, \mathcal{G})$.
We use a specific probability measure $\mu\in\mathcal{Q}$ to characterize the risk preferences of the population holistically.

The follower in the STRIPE problem is an idiosyncratic player who behaves according to the average response of the population with respect to a given probability measure $\mu\in\mathcal{Q}$. 
Accordingly, the follower's problem is described as follows:
\begin{equation}
    \min_{x\in X}U_\mu(x):=\mathbb{E}_{\theta \sim \mu}\left[ \rho_\theta [f(x,\xi)] \right],
    \label{eq:follower problem}
\end{equation}
where $U_\mu:X\rightarrow \mathbb{R}$ denotes the loss of the follower given that the distribution of RPTs is $\mu\in \mathcal{Q}$.
Problem (\ref{eq:follower problem}) extends a risk-sensitive stochastic programming problem in that the expectation over the RPTs captures the average attitude of the population towards random losses.

The leader in the Stackelberg game can choose to design $\mu$, the distribution of the RPTs.
The leader's problem admits the following form:
\begin{equation}
     \min_{\mu\in\mathcal{Q}}J(\mu, x),
     \label{eq:leader problem}
\end{equation}
where $J:\mathcal{Q}\times X \rightarrow \mathbb{R}$ denotes the loss of the leader.
The function $J$ can take different forms depending on the applications. 
One common form is the separable loss function containing two components as follows:
\begin{equation}
    J(\mu, x):=L(x)+\gamma W_1(\mu,\mu^0),
    \label{eq:leader loss func}
\end{equation}
where the function $L:X \rightarrow \mathbb{R}$ 
is assumed to be Lipschitz continuous with constant $\text{Lip}_L>0$.
In practice, the reasonable choices of the loss function $L$ include the quadratic loss $L(x)=(x-\Bar{x})^2$ representing the misalignment of the follower's action $x$ compared to an anticipated action $\Bar{x}\in X$. 
The Euclidean distance $L(x)=\dd(x,\Bar{x})$ is also a meaningful choice.
The order-$1$ Wasserstein distance $W_1(\mu,\mu^0)$ of probability measures $\mu,\mu^0\in\mathcal{Q}$ in (\ref{eq:leader loss func}) is defined by
\begin{equation*}
    W_1(\mu,\mu^0)= \inf_{\pi\in\Pi(\mu,\mu^0)} 
    \int_{\Theta\times\Theta}\dd(x,y) d\pi(x,y) ,
    \label{eq:order 1 wasserstein distance}
\end{equation*}
where $\Pi(\mu,\mu^0)$ denotes the set of all probability measures on the product space $\Theta\times\Theta$ having marginals $\mu$ and $\mu^0$.
It characterizes the difficulty of changing the distribution of RPTs to $\mu$ from the original distribution $\mu^0$. 
The parameter $\gamma>0$ balances the two terms contributing to the leader's loss.
Note that the misalignment $d(x,\Bar{x})$ depends on the probability measure $\mu$ that the leader picks, since, according to (\ref{eq:follower problem}), the action $x$ of the follower is a reaction to the leader's choice of $\mu$.

Summarizing the follower's and the leader's problems, we arrive at the (pessimistic) STRIPE problem:
\begin{equation}
\begin{aligned}
      \min_{\mu\in\mathcal{Q}}\max_{x\in X}J(\mu, x) & :=  L(x)+\gamma W_1(\mu,\mu^0), \\
      \text{s.t.} \ \ x & \in X^*(\mu)= \argmin_{x\in X}U_\mu(x):=\mathbb{E}_{\theta \sim \mu}\left[ \rho_\theta [f(x,\xi)] \right].
    \label{eq:Stackelberg problem}
\end{aligned}
\end{equation}
Problem (\ref{eq:Stackelberg problem}) can be viewed as a Stackelberg game, where the leader moves first by choosing a distribution of risk preferences $\mu$ from $\mathcal{Q}$. 
Then, the follower acts by solving (\ref{eq:follower problem}) given $\mu$. 
The leader's choice solves (\ref{eq:leader problem}) with the follower's decision from (\ref{eq:follower problem}).
Problem (\ref{eq:Stackelberg problem}) is a simple pessimistic bilevel optimization problem, since the action sets $X$ and $\mathcal{Q}$ are independent.
The pessimism captures the robustness consideration of the maximization in the leader's objective function. 
The maximization suggests that the leader chooses the worst-case scenario when the follower's optimal solution is not a singleton.
In contrast, an optimistic STRIPE can be formulated by considering $\min_{\mu\in\mathcal{Q}, x\in X}J(\mu, x)$ as the leader's problem in (\ref{eq:Stackelberg problem}).

\subsection{Solution Concepts}
\label{sec:problem:solution}
The optimal solution pair $(\mu^*, x^*)$ to problem (\ref{eq:Stackelberg problem}) is called a  Stackelberg equilibrium \cite{bacsar1998dynamic}.
Let $U^*_\mu=\inf \{U_{\mu}(x):x\in X \}$ denote the optimal value of (\ref{eq:follower problem}).
The existence of Stackelberg equilibrium depends on the continuity property of the lower level optimal value function.

\begin{lemma}
\label{lemma:lower semicontinuous}
Assume that (i) the loss function $f(\cdot, \xi)$ is convex for all $\xi\in \Xi$, and (ii) the risk measures $\rho_\theta$ for all $\theta\in\Theta$ satisfy conditions (A1) and (A2). Then, the solution mapping $X^*(\cdot)$ of the follower's problem is lower semicontinuous for all $\mu\in\mathcal{Q}$.
\end{lemma}
\begin{proof}
Under the assumptions, we know from \cite{shapiro2021lectures} that $\rho_\theta[f(x, \xi)]$ is convex in $x$. 
Since the expectation is a linear operator, $U_\mu(x)$ is convex for all $\mu\in\mathcal{Q}$.
Then, the follower's problem (\ref{eq:follower problem}) can be expressed in the following epigraph form:
\begin{equation}
\begin{aligned}
   \min_{x\in X, V\in \mathbb{R}} & V \\
   \text{s.t.} & U_\mu(x)\leq V.
   \label{eq:proof:epigraph form}
\end{aligned}
\end{equation}
Let $V^*\in\mathbb{R}$ denote the optimal solution of (\ref{eq:proof:epigraph form}).
For all $V\geq V^*$, the feasible region $\{x| U_\mu(x)-V\leq 0 \}$ takes the form of the parameterized convex constraints as in Example 5.10 of \cite{rockafellar2009variational}.
Hence the feasible region is graph convex.
Then, by Theorem 5.9 of \cite{rockafellar2009variational}, it is lower semicontinuous.
Therefore, the solution set $X^*(\mu)$ of the follower's problem is lower semicontinuous.
\qed
\end{proof}

\begin{theorem}
\label{thm:existence}
The following assertions hold: (i) There exists a global optimistic solution to the STRIPE problem (\ref{eq:Stackelberg problem});
(ii) There exists a global pessimistic solution to the STRIPE problem (\ref{eq:Stackelberg problem}), if the assumptions of Lemma \ref{lemma:lower semicontinuous} hold. 
\end{theorem}
\begin{proof}
Since the lower level problem is unconstrained and the action sets are nonempty and compact, the existence of optimistic solution follows from Theorem 5.2 of \cite{dempe2002foundations}. The existence of pessimistic solution follows from Theorem 5.3 of \cite{dempe2002foundations} under the assumptions of Lemma \ref{lemma:lower semicontinuous}.
\qed
\end{proof}

\paragraph{Robust Approximate Stackelberg Equilibrium}
In standard Stackelberg games \cite{bacsar1998dynamic}, both the Stackelberg solution and the approximate Stackelberg solution assume that the follower's problem can be solved exactly.
This assumption is restrictive for the STRIPE problem.
A closed-form expression is not often obtainable for the expectation of the random cost function in a stochastic programming problem.
The solution to the follower’s stochastic programming problem relies on the commonly-used Sample Average Approximation (SAA) technique \cite{shapiro2021lectures}, which results in an approximate solution. 
This relaxation of the follower’s problem can model the limited computational power of the follower as well as its bounded rationality in decision-making. 
The relaxation can also be anticipated by the leader, who seeks to make robust decisions that accommodate the uncertainties from the follower’s solution.  
Hence, it motivates the following solution concept for the STRIPE problem.

For a given number $\epsilon\geq0$, let $X^\epsilon:\mathcal{Q}\rightrightarrows X$ defined by $X^\epsilon(\mu):=\{x\in X| U_\mu(x)\leq U^*_{\mu}+\epsilon\}$ denote the $\epsilon$-optimal set of the follower's problem given $\mu$.

\begin{definition}
\label{def:extended stackelberg solution}
($\epsilon$-Robust $\delta$-Approximate Stackelberg Equilibrium.)
Let $\epsilon > 0$ and $\delta > 0$ be given numbers.
A pair $(\hat{x}, \hat{\mu})\in  X \times \mathcal{Q}$ is called an $\epsilon$-robust $\delta$-approximate Stackelberg equilibrium of the Stackelberg game described by the leader's problem (\ref{eq:leader problem}) and the follower's problem (\ref{eq:follower problem}), if the follower's response is an $\epsilon$-optimal solution to (\ref{eq:follower problem}) given the leader's action, \ie, 
\begin{equation}
    \hat{x}\in X^\epsilon(\hat{\mu}),
    \label{eq:epsilon approximate follower action}
\end{equation}
and the leader's action satisfies
\begin{equation}
    \sup_{x\in X^\epsilon(\hat{\mu})}J(\hat{\mu},x)
    \leq \inf_{\mu \in \mathcal{Q}}\sup_{x\in X^\epsilon(\mu)}J(\mu,x)+\delta.
    \label{eq:delta Stackelberg leader action}
\end{equation}
\end{definition}
Since $\epsilon>0$ leads to the fact that the follower's response $\hat{x}$ may be suboptimal, we refer to the parameter $\epsilon$ in Definition \ref{def:extended stackelberg solution} as the ambiguity tolerance of the leader toward the follower's responses.
The parameter $\delta$ in Definition \ref{def:extended stackelberg solution} characterizes the leader's optimality compromise in obtaining an approximate Stackelberg solution.
Definition \ref{def:extended stackelberg solution} reduces to the standard $\delta$-approximate Stackelberg solution defined in \cite{bacsar1998dynamic} if we set $\epsilon=0$, or, equivalently, if we require $\hat{x}\in X^*(\hat{\mu})$.
Note that the existence of an $\epsilon$-robust $\delta$-approximate Stackelberg equilibrium in problem (\ref{eq:Stackelberg problem}) follows directly from Theorem \ref{thm:existence}.
For the limiting behavior of the approximate Stackelberg equilibrium of Definition \ref{def:extended stackelberg solution}, we refer to \cite{lin2014solving}.

\begin{remark}
\label{remark:optimistic counterpart}
    We have considered the pessimistic perspective in Definition \ref{def:extended stackelberg solution}. The optimistic counterpart of Definition \ref{def:extended stackelberg solution} can be established by replacing condition (\ref{eq:delta Stackelberg leader action}) with $\inf_{x\in X^\epsilon(\hat{\mu})}J(\hat{\mu},x)\leq\inf_{\mu\in\mathcal{Q}, x\in X^\epsilon(\mu)}J(\mu,x)+\delta$. 
    We call the solution obtained under this definition an $\epsilon$-Optimistic $\delta$-Approximate Stackelberg Equilibrium.
    It assumes that the leader has more control over the follower's action, and it does not lead to robustness of the equilibrium solutions. We will revisit it in the contract design problem in Section \ref{sec:case:contract}. In the sequel, unless otherwise stated, we follow Definition \ref{def:extended stackelberg solution} for the game-theoretic analysis.
\end{remark}

\section{Game Analysis}
\label{sec:analysis}
In this section, we study the STRIPE problem from two different aspects.
In the first subsection, we investigate the estimation of the solution to the STRIPE problem assuming the existence of an anticipated RPT distribution and its induced anticipated follower's decision.
In the second subsection, we elaborate the connection between the leader's ambiguity tolerance and optimality compromise.

\subsection{Estimation of Solution}
\label{sec:estimation}

Let $\Bar{\mu}$ denote an RPT distribution that leads to anticipated decisions $\Bar{x}$ in the robust sense, \ie, any $\Bar{x}$ such that $\Bar{x}\in X^{\epsilon}(\Bar{\mu})$ is considered acceptable by the leader.
From the definition of the original RPT distribution $\mu^0$, we observe that the leader naturally  trades off between choosing an RPT distribution that helps approaching the anticipated outcome and one which is not too costly to achieve.
In this scenario, a convex combination of $\Bar{\mu}$ and $\mu^0$, \ie, $\mu=r\Bar{\mu}+(1-r)\mu^0$ with $r\in(0,1)$, is a reasonable candidate choice.
A natural question is how the distance between $\mu^0$ and $\Bar{\mu}$ plays the role in the solution to the STRIPE problem and the selection of the parameter $r$.

We start the analysis by referring to the quantity $W:=W_1(\Bar{\mu}, \mu^0)$ as the primitive risk perception gap of the population.
The existence of a positive primitive risk perception gap infers that there is a space for designing the distribution of RPTs, since the population's current risk perception is suboptimal from the perspective of the leader.
This eliminates the trivial case where the existing distribution of the RPTs $\mu^0$ and the anticipated one $\Bar{\mu}$ coincide with each other.
Hence, we assume throughout the discussion that $W_1(\Bar{\mu}, \mu^0)>0$.

Our approach is based on the distances of the $\epsilon$-optimal sets when the choice of $\mu$ varies.
Hence, we adopt the following standard notations.
We use $\DD(a, A)$ to denote the distance from a point $a$ to a set $A$.
For two compact sets $A$ and $B$, we express the deviation of $A$ from $B$ as:
\begin{equation}
    \mathbb{D}(A,B):=\sup_{a\in A}\DD(a, B).
    \label{eq:deviation of sets}
\end{equation}
We focus on the quantity $\mathbb{D}(X^\epsilon(\mu),X^\epsilon(\Bar{\mu}))$ and refer to it as the disappointment of the follower's action in the worst-case scenario when the leader's choice of RPT distribution deviates from $\Bar{\mu}$.
The reason lies in the following.
From Definition \ref{def:extended stackelberg solution}, we observe that $X^\epsilon(\Bar{\mu})$ is the set of follower's actions which the leader can accept according to the $\epsilon$-robustness.
Then, the deviation $\mathbb{D}(X^\epsilon(\mu),X^\epsilon(\Bar{\mu}))$ measures the largest possible distance from $X^\epsilon(\mu)$, which contains all the potential follower's actions obeying (\ref{eq:epsilon approximate follower action}) given the leader's choice $\mu$, to the acceptable follower's action in $X^\epsilon(\Bar{\mu})$.

The following result provides an estimation of the $\epsilon$-robust $\delta$-approximate Stackelberg solution of problem (\ref{eq:Stackelberg problem}) using the primitive risk perception gap $W_1(\Bar{\mu}, \mu^0)$, assuming the knowledge of $\Bar{\mu}$.

\begin{proposition}
\label{prop:bounds on extended solution}
Suppose that the second-order growth condition holds for the follower's problem , \ie, there exists $\iota>0$, such that 
\begin{equation}
    U_{\Bar{\mu}}(x)\geq\min_{x\in X}U_{\Bar{\mu}}(x)+\iota\cdot \DD(x,X^*(\Bar{\mu}))^2, \forall x\in X,
    \label{eq:second-order growth condition}
\end{equation}
then, 
\begin{equation}
     \mathbb{D}(X^\epsilon(\mu),X^\epsilon(\Bar{\mu}))
      \leq 
      \sqrt{\frac{3}{\iota}W_1(\mu, \Bar{\mu})}.
    \label{eq:bound 1 on extended solution}
\end{equation}
Furthermore, for $\mu=r\Bar{\mu}+(1-r)\mu^0$ and $x\in X^\epsilon(\mu)$, the performance reduction observed by the leader described by $\sup_{x\in X^\epsilon(\mu)}\inf_{\Bar{x}\in X^\epsilon(\Bar{\mu})}|L(x)-L(\Bar{x})|$ satisfies
\begin{equation}
    \sup_{x\in X^\epsilon(\mu)}\inf_{\Bar{x}\in X^\epsilon(\Bar{\mu})}|L(x)-L(\Bar{x})|
    \leq
    \text{Lip}_L\sqrt{\frac{3}{\iota}(1-r)W}.
    \label{eq:performance reduction}
\end{equation}

\end{proposition}
\begin{proof}
We prove the assertion by connecting $\mathbb{D}(X^\epsilon(\mu),X^\epsilon(\Bar{\mu}))$ with $W_1(\mu, \Bar{\mu})$.
We first make the following observation. 
The relation $\mathbb{D}(X^\epsilon(\mu),X^\epsilon(\Bar{\mu}))\leq \alpha$ is equivalent to $D(x,X^\epsilon(\Bar{\mu}))\leq \alpha$ for all $ x\in X^\epsilon(\mu)$.
The later is equivalent to its contraposition which states that if $D(x,X^\epsilon(\Bar{\mu}))>\alpha$ then $x \notin X^\epsilon(\mu)$.
This contraposition is again equivalent to requiring that if for all $x\in X$, the condition $D(x,X^\epsilon(\Bar{\mu}))>\alpha$ holds, then there exists $x'\in X$ such that $U_{\mu}(x)-U_{\mu}(x')>\epsilon$.
Now, suppose that $\sup_{x\in X}|U_\mu(x)-U_{\Bar{\mu}}(x)|\leq \beta$ holds.
Then, for all $x\in X$ such that $D(x,X^\epsilon(\Bar{\mu}))>\alpha$ and $\forall x^\epsilon\in X^\epsilon(\Bar{\mu})$, the following holds:
\begin{equation}
        U_\mu(x)-U_\mu(x^\epsilon)\geq U_{\Bar{\mu}}(x)-U_{\Bar{\mu}}(x^\epsilon)-2\beta.
        \label{eq:proof:2beta ineq}
\end{equation}
Let
\begin{equation*}
    R^\epsilon(\alpha)=\inf_{d(x,X^\epsilon(\Bar{\mu}))>\alpha}
    \left( U_{\Bar{\mu}}(x)-U^*_{\Bar{\mu}} \right)-\epsilon.
\end{equation*}
Clearly, $R^\epsilon(\alpha)>0$.
Choose $\beta=\frac{1}{3}R^\epsilon(\alpha)$, then the inequality (\ref{eq:proof:2beta ineq}) leads to:
\begin{equation*}
    \begin{aligned}
          U_\mu(x)-U_\mu(x^\epsilon)& \geq R^\epsilon(\alpha)+\epsilon-\frac{2}{3}R^\epsilon(\alpha) \\
          & \geq \epsilon+\frac{1}{3}R^\epsilon(\alpha) \\
          & > \epsilon.
    \end{aligned}
\end{equation*}
Hence, $\mathbb{D}(X^\epsilon(\mu),X^\epsilon(\Bar{\mu}))\leq\alpha$ if the condition $\sup_{x\in X}|U_\mu(x)-U_{\Bar{\mu}}(x)|\leq \beta$ holds.
Moreover, the second-order growth condition (\ref{eq:second-order growth condition}) leads to 
\begin{equation*}
    U_{\Bar{\mu}}(x)-U^*_{\Bar{\mu}}\geq \iota \cdot D(x,X^*(\Bar{\mu}))^2, \forall x\in X.
\end{equation*}
Since $X^*(\Bar{\mu})\subset X^\epsilon(\Bar{\mu})$, we know that $D(x, X^*(\Bar{\mu})) \geq D(x, X^\epsilon(\Bar{\mu}))$.
Therefore, for all $x\in X$ such that $D(x,X^\epsilon(\Bar{\mu}))\geq \alpha$, the following relation holds:
\begin{equation*}
          U_{\Bar{\mu}}-U^*_{\Bar{\mu}} \geq \iota\cdot D(x, X^*(\Bar{\mu}))^2
          \geq \iota\cdot D(x, X^\epsilon(\Bar{\mu}))^2 
          \geq \iota \cdot \alpha^2.
\end{equation*}
Then, we obtain $R^\epsilon(\alpha)+\epsilon=\iota\cdot \alpha^2$. 
Choose $\alpha=\sqrt{\frac{3}{\iota}\sup_{x\in X}|U_{\Bar{\mu}}(x)-U_\mu(x)|}$,
Then, for all $x\in X$ such that $D(x, X^\epsilon(\Bar{\mu}))\geq \alpha$ and for all $x^\epsilon \in X^\epsilon(\Bar{\mu})$, the following holds:
\begin{equation*}
    \begin{aligned}
          U_\mu(x)-U_\mu(x^\epsilon)& \geq U_{\Bar{\mu}}(x)-U_{\Bar{\mu}}(x^\epsilon)-2\sup_{x\in X}|U_{\Bar{\mu}}(x)-U_\mu(x)| \\
          & \geq R^\epsilon(\alpha)+\epsilon-2\sup_{x\in X}|U_{\Bar{\mu}}(x)-U_\mu(x)| \\
          & \geq \sup_{x\in X}|U_{\Bar{\mu}}(x)-U_\mu(x)|\\
          & \geq 0.
    \end{aligned}
\end{equation*}
Therefore, 
\begin{equation*}
     \mathbb{D}(X^\epsilon(\mu),X^\epsilon(\Bar{\mu}))\leq \sqrt{\frac{3}{\iota}\sup_{x\in X}|U_{\Bar{\mu}}(x)-U_\mu(x)|}.
\end{equation*}
Let $H(\Theta)$ denote the set of all locally Lipschitz functions  $h:\Theta\rightarrow \mathbb{R}$.
Then, we have the following inequality:
\begin{equation*}
    \begin{aligned}
          \sup_{x\in X}|U_{\Bar{\mu}}(x)-U_\mu(x)| & =\sup_{x\in X}| \int_{\Theta}F(x,\theta)d\Bar{\mu}(\theta)-\int_{\Theta}F(x,\theta)d\mu(\theta) | \\
          &
          \leq  \sup_{h \in H(\Theta)}|\int_{\Theta}hd\Bar{\mu}-\int_{\Theta}hd\mu| \\
        &=  W_1(\Bar{\mu}, \mu).
    \end{aligned}
\end{equation*}
Hence, we arrive at (\ref{eq:bound 1 on extended solution}).
The inequality (\ref{eq:performance reduction}) is obtained by applying the convexity of the order-$1$ Wasserstein distance and the Lipschitz property of $L(\cdot)$.
This completes the proof.
\qed
\end{proof}
The inequality (\ref{eq:performance reduction}) 
allows us to estimate the approximate solution to the STRIPE problem given a candidate measure $\mu$ since the design cost measured by $W_1(\mu,\mu^0)$ can be obtained directly.

\begin{remark}
The consideration of $\mathbb{D}(X^\epsilon(\mu),X^\epsilon(\Bar{\mu}))$ and $\sup_{x\in X^\epsilon(\mu)}\inf_{\Bar{x}\in X^\epsilon(\Bar{\mu})}|L(x)-L(\Bar{x})|$ in Proposition \ref{prop:bounds on extended solution} is in line with Definition \ref{def:extended stackelberg solution}.
The stochastic nature of the follower's problem (\ref{eq:follower problem}) makes the set-valued map $X^\epsilon(x)$ a meaningful target to investigate.
Both of $\mathbb{D}(X^\epsilon(\mu),X^\epsilon(\Bar{\mu}))$ and $\sup_{x\in X^\epsilon(\mu)}\inf_{\Bar{x}\in X^\epsilon(\Bar{\mu})}|L(x)-L(\Bar{x})|$ take into account the $\epsilon$-robustness.
\end{remark}
\begin{remark}
The second-order growth condition (\ref{eq:second-order growth condition}) has been widely used in the stability analysis of optimization problems. Early works include \cite{bonnans1995quadratic,shapiro1994quantitative}. 
The result in Proposition \ref{prop:bounds on extended solution} builds on the results of \cite{liu2013stability,pichler2018quantitative} and extends to $\epsilon$-optimal sets in terms of Definition \ref{def:extended stackelberg solution}.
\end{remark}

\subsection{Connection Between Ambiguity Tolerance and Optimality Compromise}
\label{sec:asymmetry}

The equilibrium strategies obtained from (\ref{eq:epsilon approximate follower action}) and (\ref{eq:delta Stackelberg leader action}) depend heavily on the selections of ambiguity tolerance and optimality compromise.
On the one hand, ambiguity tolerance and optimality compromise are the design parameters in Definition \ref{def:extended stackelberg solution} that one needs to choose before solving STRIPE.
On the other hand, there is an interdependency between them.
We observe from (\ref{eq:delta Stackelberg leader action}) that an increase in ambiguity tolerance leads to an enlarged feasible set of $\sup_{x\in X^\epsilon (\mu)}J(\mu,x)$.
Then, for $\forall \mu\in\mathcal{Q}$ and $\epsilon_1\leq \epsilon_2$, one observes that $\sup_{x_1\in X^{\epsilon_1}(\mu)}J(\mu, x_1)\leq \sup_{x_2\in X^{\epsilon_2}(\mu)}J(\mu, x_2)$.
Hence, the value of the game satisfies $\inf_{\mu\in \mathcal{Q}}\sup_{x_1\in X^{\epsilon_1}(\mu)}J(\mu, x_1)\leq \inf_{\mu\in \mathcal{Q}}\sup_{x_2\in X^{\epsilon_2}(\mu)}J(\mu, x_2)$ for $\epsilon_1\leq\epsilon_2$.
Suppose that the leader compromises a fixed percentage of the value of the game. 
Then, her optimality compromises $\delta_1$ and $\delta_2$ under ambiguity tolerances $\epsilon_1$ and $\epsilon_2$ should satisfy $\delta_1\leq \delta_2$.
In the case of $\epsilon$-Optimistic $\delta$-Approximate Stackelberg Equilibrium stated in Remark \ref{remark:optimistic counterpart}, we arrive at the opposite relation.

In the following, we investigate the connections between the ambiguity tolerance and the optimality compromise to aid their selections. 
In particular, our goal is to find an upper estimate of the optimality compromise induced by one's choice of ambiguity tolerance.
We leverage mathematical tools from set-valued analysis and study the behavior of the mapping $X^\epsilon(\cdot)$.

For a general set-valued map $S:A\rightrightarrows B$ with graph $\text{gph}S:=\{(a,b)\in A\times B| b\in S(a)\}$, its inverse map $S^{-1}:B\rightrightarrows A$ is defined by $\text{gph}S^{-1}:=\{(b,a)\in B\times A | (a,b)\in \text{gph}S \}$.
Note that the inverse of a set-valued is always well-defined \cite{aubin2009set,mordukhovich2006variational}.
To study the leader's choice when a solution of the follower is given, we leverage the following Lipschitzian property defined for set-valued maps.

\begin{definition}
\label{def:metric regularity}
(Metric regularity.)
A set-valued map $S:A\rightrightarrows B$ is said to be metrically regular around a point $(\Bar{a},\Bar{b})\in \text{gph}S$ with constant $r\in\mathbb{R}_+$, if there are neighborhoods $\mathcal{N}_A$ of $\Bar{a}$ and $\mathcal{N}_B$ of $\Bar{b}$ such that
\begin{equation}
    D(a,S^{-1}(b))\leq r\cdot D(b,S(a)), \forall a\in \mathcal{N}_A , \forall b\in \mathcal{N}_B.
    \label{eq:metric regularity}
\end{equation}
\end{definition}
We denote by $\text{reg}S(\Bar{a},\Bar{b})$ the infimum of the constant $r$ which satisfies condition (\ref{eq:metric regularity}) among all choices of neighborhoods $\mathcal{N}_A$ and $\mathcal{N}_B$.
Condition (\ref{eq:metric regularity}) is a Lipschitz-type inequality for the inverse map $S^{-1}$.
It is well-known that a set-valued map $S$ is metrically regular if and only if its inverse set-valued map $S^{-1}$ has the Aubin property \cite{rockafellar2009variational,mordukhovich2006variational}, which is an extension of the Lipschitz property of single-valued maps to set-valued maps \cite{aubin2009set}.

Necessary and sufficient conditions guaranteeing the metric regularity of set-valued maps $S:A\rightrightarrows B$ between Banach space has been derived in, for example, \cite{artacho2010metric,mordukhovich2006variational}.
The results are based on the essential surjectivity requirement of the linear operator between the spaces described in Theorem 3.3 of \cite{artacho2010metric}.
In our case, as we will see later, this assumption does not hold.
Instead, we will focus on the restrictive metric regularity \cite{mordukhovich2004restrictive} which is the same Lipschitz-type property of $S$ as stated in (\ref{eq:metric regularity}) but restricted to the image $S(A)\subset B$.
We will show that the $\epsilon$-solution set $X^\epsilon(\cdot)$ of the follower's problem is restrictive metrically regular. 

Consider the following auxiliary problem.
Let $\mathcal{M}(\Theta)$ denote the set of all finite signed Borel measures on $\Xi$.
Consider problem (\ref{eq:follower problem}) with the set of parameters $\Tilde{\mu}\in \mathcal{M}(\Theta)$,  \ie,
\begin{equation}
     \min_{x\in X}\Tilde{U}_{\Tilde{\mu}}(x):=\mathbb{E}_{\theta \sim \Tilde{\mu}}\left[ \rho_\theta [f(x,\xi)] \right].
    \label{eq:extended follower problem}
\end{equation}
Let $\Tilde{U}^*_{\Tilde{\mu}}$ and $\Tilde{X}^\epsilon(\Tilde{\mu})$ denote the counterparts of $U^*_\mu$ and $X^\epsilon(\mu)$ when the measures $\Tilde{\mu}$ are Borel measures from $\mathcal{M}(\Theta)$.
Accordingly, the order-$1$ Wasserstein distance  between elements of $\mathcal{M}(\Theta)$ can be defined similarly to (\ref{eq:order 1 wasserstein distance}).

Here, we note the fact that the Wasserstein distance of order $1$ can be extended to a norm when $\mathcal{M}(\Theta)$ is of interest. 
Let $\mu_+$ and $\mu_-$ denote the positive and negative variations of $\mu\in\mathcal{M}(\Theta)$, respectively.
Let $\text{Var}(\mu)=|\mu|(\Theta)$ denote the total variation of $\mu$, where $|\mu|=\mu_+ + \mu_-$.
Let $\mathcal{M}_0(\Theta)$ be the set of measures defined on $\Xi$ such that any $\mu\in\mathcal{M}_0(\Theta)$ satisfies $\mu(\Xi)=0$.
Let $\Psi_{\mu}$ denote the family of all nonnegative measures associated with $\mu\in\mathcal{M}_0(\Theta)$ defined on $\mathcal{M}(\Theta\times\Theta)$, such that for any Borel set on $A\in\Theta$, $\Psi(\Theta, A)-\Psi(A,\Theta)=\mu(A)$.
The Kantorovich-Rubinstein (KR) norm introduced in \cite{kantorovich1957functional} is defined as follows:
\begin{equation}
    ||\mu||^0_{\text{KR}}=\inf \Big\{ \int_{\Xi\times\Xi} d(x,y)d\psi(x,y): \psi\in\Psi_{\mu} \Big\}.
    \label{eq:KR norm}
\end{equation}
Note that the norm in (\ref{eq:KR norm}) is closely related to $W_1(\mu,\nu)$. Indeed, the quantity $\Psi(A_1,A_2)$ has the interpretation of mass transportation from set $A_1$ to set $A_2$. Hence, a measure $\psi\in\Psi_\mu$ stands for a mass transfer from initial distribution $\mu_-$ to target distribution $\mu_+$.
In \cite{hanin1992kantorovich}, the KR norm is extended to every measure $\mu\in\mathcal{M}(\Theta)$ as follows:
\begin{equation}
    ||\mu||_{\text{KR}}=\inf\{||\nu||^0_{\text{KR}}+\text{Var}(\mu-\nu): \nu\in\mathcal{M}_0(\Theta) \}.
    \label{eq:extended KR norm}
\end{equation}
Note that $||\cdot||_{\text{KR}}$ coincides with $||\cdot||^0_{\text{KR}}$ when the measure of interest is in $\mathcal{M}_0(\Theta)$.
Let $\Tilde{W}_1(\mu,\nu)=||\mu-\nu||_{\text{KR}}$ denote the distance for $\mu,\nu\in\mathcal{M}(\Theta)$ induced by the extended KR norm.

It is discussed in \cite{hanin1992kantorovich} that the space of all finite signed Borel measures on $\Theta$ with the extended KR norm is isometrically isomorphic to the space $lip(\Theta,d)^*$ \cite{johnson1974lipschitz}. 
This result is established by considering the dual optimal transport problem \cite{villani2009optimal} associated with the Wasserstein distance of order $1$, which admits the form of finding a Lipschitz cost function that maximizes the difference between the total cost induced by the initial measure and the total cost induced by the target measure.
One of the consequences of this isomorphism is that the space of all finite signed Borel measures on $\Theta$ with the extended KR norm is a Banach space when the underlying space $(\Theta, d)$ is complete. 
We will leverage this property when we discuss the main result of this section.
More discussions on the Wasserstein spaces can be found in \cite{villani2009optimal,ambrosio2008gradient}.

Let $G:\mathcal{M}(\Theta)\times X\rightarrow \mathbb{R}$ be defined as $G(\Tilde{\mu},x)=-\int_{\Theta}F(x,\theta)d\Tilde{\mu}(\theta)+\Tilde{U}^*_{\Tilde{\mu}}$.

\begin{lemma}
\label{lemma:metric}
Let  $\mathcal{N}_{\Tilde{\mu}_*}\times \mathcal{N}_{x_*}$ denote a neighborhood of $(\Tilde{\mu}_*,x_*)\in \mathcal{M}(\Theta)\times X$.
Suppose that the following conditions hold:
\\
(M1) $X$ is compact and convex;
\\
(M2) $f(\cdot,\xi)$ is Lipschitz on $\mathcal{N}_{x_*}$ for all $\xi\in\Xi$;
\\
(M3) $f(x,\cdot)$ is bounded for all $x\in X$; 
\\
(M4) $f(\cdot,\xi)$ is convex for all $\xi \in \Xi$;
\\
(M5) $\rho_\theta$ satisfy conditions (A2) and (A2) for all $\theta\in\Theta$.
\\
Then, the set-valued map $\Tilde{X}^\epsilon$ is restrictive metrically regular around $(\Tilde{\mu}_*,x_*)$, \ie, for any point $(\Tilde{\mu}_1,x_1)\in (\mathcal{N}_{\Tilde{\mu}_*}\times \mathcal{N}_{x_*}) \cap \text{gph}\Tilde{X}^\epsilon$ and $x_2\in \mathcal{N}_{x_*}$ there exists a constant $\mathtt{M}>0$ and $\Tilde{\mu}_2\in (\Tilde{X}^\epsilon)^{-1}(x_2) $ such that the following inequality holds:
\begin{equation}
    \Tilde{W}_1(\Tilde{\mu}_1,\Tilde{\mu}_2)\leq \mathtt{M} \cdot d(x_1,x_2).
    \label{eq:metric regularity bound}
\end{equation}
\end{lemma}
\begin{proof}
To obtain the metric regularity property of $\Tilde{X}^\epsilon$, we consider its reformulation as the following variational system:
\begin{equation}
\begin{aligned}
    \widetilde{X}^{\epsilon}(\Tilde{\mu})
    &=\Big\{ x\in X \big\arrowvert \int_{\Theta}F(x,\theta)d\Tilde{\mu}(\theta)\leq \widetilde{U}^*_{\Tilde{\mu}}+\epsilon \Big\} \\
    &=\Big\{ x\in X \big\arrowvert 0\in G(\Tilde{\mu},x)+\epsilon+\mathbb{R}_- \Big\}.
    \label{eq:proof:variational sys reformulation}
\end{aligned}
\end{equation}
According to the discussions on the Wasserstein distances in Section \ref{sec:asymmetry}, we know that $\widetilde{X}^{\epsilon}$ is a set-valued map defined between Banach spaces.
Next, we will follow \cite{artacho2010metric} for the discussion of the metric regularity of the parametric variational system (\ref{eq:proof:variational sys reformulation}).
Since we assume that condition (M2) holds and $\nu^*_\theta(\xi)$ is a density function for all $\theta\in\Theta$, $F(\cdot,\theta)=\int_{\Xi} f(x,\xi)\nu^*_\theta (\xi)dP(\xi)$ is Lipschitz on $\mathcal{N}_{x_*}$ for all $\theta\in\Theta$.
This also implies that $F(\cdot,\theta)$ is lower semicontinuous.
The lower semicontinuity of $F(\cdot,\theta)$ and condition (M1) make the set $\arg\min_{x\in X}U_\mu(x)$ nonempty.
Let $\Tilde{\mu}_1, \Tilde{\mu}_2 \in \mathcal{M}(\Theta)$.
Let $x_1^*\in\arg\min_{x\in X}\Tilde{U}_{\Tilde{\mu}_1}(x)$ and $x_2^*\in\arg\min_{x\in X}\Tilde{U}_{\Tilde{\mu}_2}(x)$.
Then, we have the following estimate of the distance of optimal values:
\begin{equation}
\begin{aligned}
    |\Tilde{U}^*_{\Tilde{\mu}_1}-\Tilde{U}^*_{\Tilde{\mu}_2}|
    & \leq \max \Big\{ \int_{\Theta}F(x_1^*,\theta)(\Tilde{\mu}_2-\Tilde{\mu}_1)(d\theta), \int_{\Theta}F(x_2^*,\theta)(\Tilde{\mu}_1-\Tilde{\mu}_2)(d\theta)\Big\} \\
    & \leq \sup_{h\in H(\Theta)} |\int_{\Theta}h(\theta)(\Tilde{\mu}_1-\Tilde{\mu}_2)(d\theta)|,
    \label{eq:proof:lipschitz condition for optimal value}
\end{aligned}
\end{equation}
where $H(\Theta)$ denotes the set of all locally Lipschitz functions  $h:\Theta\rightarrow \mathbb{R}$.
In the second inequality of (\ref{eq:proof:lipschitz condition for optimal value}), we observe that $\sup_{h\in H(\Theta)} |\int_{\Theta}h(\theta)(\mu_1-\mu_2)(d\theta)|$ is the dual representation of the Wasserstein distance of order 1 \cite{villani2009optimal}.
Hence, we obtain that the optimal value function $\Tilde{U}^*_{\mu}$ is Lipschitz, \ie, $|\Tilde{U}^*_{\mu_1}-\Tilde{U}^*_{\mu_2}| \leq \Tilde{W}_1(\mu_1,\mu_2)$.
Condition (M3) implies that $\int_{\Theta}F(x,\theta)d\Tilde{\mu}(\theta)$ is Lipschitz in $\Tilde{\mu}$.
Thus, we conclude with the discussions above that
$G$ is Lipschitz on the neighborhood $\mathcal{N}_{\Tilde{\mu}_*}\times \mathcal{N}_{x_*}$.
Next, we observe from the linearity property that $\int_{\Theta}F(x,\theta)d\Tilde{\mu}(\theta)$ is strictly partially differentiable with respect to $\Tilde{\mu}$ on $\mathcal{M}(\Theta)$ with derivative $[F(x,\theta)]_{\theta\in\Theta}$.
Together with condition (M4), we obtain that $G$ is strictly differentiable.
On one hand, if the partial derivative $\nabla_{\Tilde{\mu}}G(\Tilde{\mu},x):\mathcal{M}(\Theta)\rightarrow \mathbb{R}$ is surjective, then we conclude the metric regularity of $\Tilde{X}^\epsilon$ using Corollary 3.5 in \cite{artacho2010metric} and obtain exact bounds.
However, this surjectivity condition is clearly violated, since $F(x,\theta)$ is constant given $x\in \mathcal{N}_{x_*}$.
On the other hand, observe from the proof of Lemma \ref{lemma:lower semicontinuous} that under conditions (M4) and (M5), problem (\ref{eq:extended follower problem}) is a convex problem.
Then, by Danskin's theorem, the derivative of $\Tilde{U}^*_{\Tilde{\mu}}$ with respect to $\Tilde{\mu}$ is given by $\frac{\partial}{\partial \Tilde{\mu}}(\Tilde{U}_{\Tilde{\mu}}(\Tilde{x}^*))=[\rho_\theta[f(\Tilde{x}^*,\xi)]_{\theta\in\Theta}$ for all $\Tilde{x}^*\in \Tilde{X}^*(\Tilde{\mu})$, which is constant with respect to $\Tilde{\mu}$.
Then, $\Tilde{U}^*_{\Tilde{\mu}}$ is continuously differentiable and hence strictly differentiable \cite{clarke1990optimization}.
Since the maximum of (\ref{eq:extended follower problem}) is attained, the set $\{[\rho_\theta[f(\Tilde{x}^*,\xi)]_{\theta\in\Theta}|\Tilde{x}^*\in \Tilde{X}^*(\Tilde{\mu}), \Tilde{\mu}\in\mathcal{M}(\Theta)\}$ is closed.
Hence, the set of the strict derivatives of $G$ with respect to $\Tilde{\mu}$ is closed.
Therefore, according to Theorem 2.10 of \cite{mordukhovich2004restrictive}, $\Tilde{X}^\epsilon$ has the restrictive metric regularity property around $(\Tilde{\mu}_*,x_*)$.
Finally, the inequality (\ref{eq:metric regularity bound}) holds due to the equivalent expression presented in \cite{klatte2006nonsmooth}.
\qed
\end{proof}

Note that metric regularity does not always hold in variational systems \cite{mordukhovich2008failure,artacho2010metric}. The reason why the metric regularity holds lies in the fact that the field operator $\mathbb{R}_-$ in (\ref{eq:proof:variational sys reformulation}) is not monotone in the sense of \cite{mordukhovich2008failure}.

Since problem (\ref{eq:follower problem}) can be considered as problem (\ref{eq:extended follower problem}) restricted onto $\mathcal{Q}\in\mathcal{M}(\Theta)$, the next result follows directly from Lemma \ref{lemma:metric}.
\begin{corollary}
\label{coro:metric}
Let the assumptions in Lemma \ref{lemma:metric} hold. Suppose that leader's choice of RPT distribution is $\mu_1\in\mathcal{Q}$ and her anticipated follower's $\epsilon$-approximate optimal solution is $x_1\in X^\epsilon(\mu_1)$.
Then, for a given follower's response $x_2$, there exists a constant $\mathtt{M}>0$ and an RPT distribution $\mu_2$, such that the following holds:
\begin{equation}
    W_1(\mu_1,\mu_2)\leq \mathtt{M} \cdot d(x_1,x_2),
    \label{eq:metric regularity bound RPT}
\end{equation}
provided that $d(x_1,x_2)$ is sufficiently small.
\end{corollary}
Consider a leader's choice of RPT distribution $\mu_1$ and an anticipated follower's response $x_1\in X^\epsilon(\mu_1)$.
Since the leader adopts ambiguity tolerance $\epsilon$, she must take into account all of the follower's responses $x_2\in X^{\epsilon}(\mu_1)$. 
Note that there exists multiple RPT distributions $\mu_2\in\mathcal{Q}$ which can lead to a given  $x_2\in X^{\epsilon}(\mu_1)$.
Then, Corollary \ref{coro:metric} states that the distance between $\mu_1$ and $\mu_2$ is upper bounded by the distance between $x_1$ and $x_2$ up to a constant.
Before we present the next result that upper bounds the optimality compromise using the ambiguity tolerance, we first discuss the following interpretation of Corollary \ref{coro:metric}.

As we have discussed in Section \ref{sec:problem:stackelberg} that the leader in STRIPE designs the RPT distribution of the population.
We have assumed that the RPT distribution of the population will become the one chosen by the leader.
However, the leader can only verify whether her choice of RPT distribution is adopted by the population or not through the follower's response.
Then, the assertions in Corollary \ref{coro:metric} can be equivalently understood as the fact that the population's RPT distribution after the risk preference design is not too far away from the leader's choice when the follower's response is close to the leader's anticipation.

\begin{theorem}
\label{thm:optimality compromise}
Let the assumptions in Lemma \ref{lemma:metric} and the second-order growth condition (\ref{eq:second-order growth condition}) hold.
Suppose that the leader's ambiguity tolerance towards the follower's approximate solutions is $\epsilon>0$. 
Then, the leader's optimality compromise $\delta>0$ induced by her ambiguity tolerance satisfies
\begin{equation}
    \delta\leq \sqrt{\frac{\epsilon}{\iota}}\cdot (\text{Lip}_L + \gamma \mathtt{M}).
    \label{eq:upper estimate of optimality compromise}
\end{equation}
\end{theorem}
\begin{proof}
Suppose that $\mu^*\in\mathcal{Q}$ denotes the optimal choice of RPT distribution of the leader, and any $x^*\in X^*(\mu^*)$ is an acceptable optimal solution from the follower given $\mu^*$.
To derive the optimality compromise of the leader, we consider a pair $(\hat{\mu},\hat{x})$ where $\hat{\mu}$ is an RPT distribution satisfying $X^\epsilon(\mu^*) \cap X^\epsilon(\hat{\mu})\neq \emptyset$ and $\hat{x}\in X^\epsilon(\mu^*) \cap X^\epsilon(\hat{\mu})$ represents a follower's response which meets the ambiguity tolerance $\epsilon$ from the perspective of the leader.
Then, the optimality compromise $\delta$ induced by the ambiguity tolerance $\epsilon$ can be measured by the minimum difference between the leader's costs under the worst possible pair $(\hat{\mu}, \hat{x})$ and under an optimal pair $(\mu^*,x^*)$, \ie, $\delta= \sup_{(\hat{\mu},\hat{x})}\inf_{(\mu^*,x^*)} J(\hat{\mu},\hat{x})-J(\mu^*,x^*)$ where the supremum is taken over all $(\hat{\mu},\hat{x})$ pairs satisfying $X^\epsilon(\mu^*) \cap X^\epsilon(\hat{\mu})\neq \emptyset$ and $\hat{x}\in X^\epsilon(\mu^*) \cap X^\epsilon(\hat{\mu})$ and the infimum is taken over all $(\mu^*,x^*)$ pairs such that $x^*\in X^*(\mu^*)$.
Since any point from $X^*(\mu^*)$ is optimal given $\mu^*$, and any point from $X^\epsilon(\mu^*)$ meets the ambiguity tolerance, we know from the definition of the deviation of sets (\ref{eq:deviation of sets}) that $\mathbb{D}(X^\epsilon(\mu^*),X^*(\mu^*))$ is a proper upper-estimate of the distance between the points $x$ and $\hat{x}$ of interest, \ie, $\sup_{\hat{x}\in X^\epsilon(\mu^*)}\inf_{x^*\in X^*(\mu^*)}||x^*-\hat{x}||= \mathbb{D}(X^\epsilon(\mu^*),X^*(\mu^*))$.
Recall that the second-order growth condition is
\begin{equation*}
    U_{\mu^*}(x)\geq \min_{x'\in X}U_{\mu^*}(x')+\iota\cdot D(x,X^*(\mu^*))^2, \ \ \forall x\in X, 
\end{equation*}
which implies that 
\begin{equation*}
    D(x,X^*(\mu^*))\leq \sqrt{\frac{1}{\iota}(U_{\mu^*}(x)-U_{\mu^*}(x^*))}, \ \ \forall x\in X.
\end{equation*}
In particular, for $x\in X^\epsilon(\mu^*)=\{x|U_{\mu^*}(x)-U_{\mu^*}(x^*)\leq \epsilon\}$, we obtain that $D(x,X^*(\mu^*))\leq \sqrt{\frac{\epsilon}{\iota}}, \ \ \forall x\in X^\epsilon(\mu^*)$, which is equivalent to $\mathbb{D}(X^\epsilon(\mu^*),X^*(\mu^*))\leq \sqrt{\frac{\epsilon}{\iota}}$.
From Corollary \ref{coro:metric}, we observe that $W_1(\mu^*,\hat{\mu})\leq \mathtt{M}\cdot ||x^*-\hat{x}||$.
Therefore, we obtain
\begin{equation*}
\begin{aligned}
    \delta
    &= \sup_{(\hat{\mu},\hat{x})}\inf_{(\mu^*,x^*)} J(\hat{\mu},\hat{x})-J(\mu^*,x^*) \\
    & =\sup_{(\hat{\mu},\hat{x})} \inf_{(\mu^*,x^*)} L(\hat{x})-L(x^*)+\gamma W_1(\hat{\mu},\mu^*) \\
    & \leq \sup_{\hat{x}\in X^\epsilon(\mu^*)}\inf_{x^*\in X^*(\mu^*)} \text{Lip}_L ||\hat{x}-x^*||+\gamma\mathtt{M}||\hat{x}-x^*|| \\
    & \leq \mathbb{D}(X^\epsilon(\mu^*),X^*(\mu^*))(\text{Lip}_L+\gamma\mathtt{M}) \\
    & \leq \sqrt{\frac{\epsilon}{\iota}}\cdot(\text{Lip}_L+\gamma\mathtt{M}).
    \label{eq:proof:difference of leader loss}
\end{aligned}
\end{equation*}
This completes the proof.
\qed
\end{proof}

The upper bound of the leader's optimality compromise given by (\ref{eq:upper estimate of optimality compromise}) is an estimate of the leader's loss induced by her ambiguity tolerance $\epsilon$.
This estimate provides a way for selecting the parameters of the robust approximate Stackelberg equilibrium of Definition \ref{def:extended stackelberg solution}.
Note that (\ref{eq:upper estimate of optimality compromise}) only measures the change in the leader's loss resulting from the leader's tolerance of the $\epsilon$-optimal responses from the follower. 
This fact suggests that, apart from the optimality tolerance obtained from (\ref{eq:upper estimate of optimality compromise}), the leader may also consider relaxed solutions to the upper level problem because of the inexactness originating from other sources, such as the inaccuracy of the original RPT distribution $\mu^0$ and the modeling error of using a cost function $L(\cdot)$.

\section{Single-level Reformulation of STRIPE}
\label{sec:reformulation}
The STRIPE problem (\ref{eq:Stackelberg problem}) is in general challenging to solve both analytically and computationally due to its bilevel structure.
In this section, we study the reformulations of (\ref{eq:Stackelberg problem}) to obtain its data-driven solutions.

Recall that the follower's cost function given $\mu\in\mathcal{Q}$ is 
\begin{equation*}
    U_\mu(x)=\mathbb{E}_{\theta\sim \mu}[\rho_\theta[f(x,\xi)]].
    \label{eq:equivalent risk measure}
\end{equation*}
Define $\chi_\mu:\mathcal{Z}\rightarrow \mathbb{R}$ by
\begin{equation}
    \chi_\mu[f(x,\xi)]=U_\mu(x)=\mathbb{E}_{\theta\sim \mu}[\rho_\theta[f(x,\xi)]].
    \label{eq:equivalent risk measure}
\end{equation}
The function $\chi_\mu[\cdot]$ is a convex combination of coherent risk measures $\rho_\theta$ for $\theta\in \Theta$.
Hence, $\chi_\mu[\cdot]$ is itself a coherent risk measure \cite{acerbi2002spectral}.
This fact indicates that the expression of the follower's cost function (\ref{eq:dual representation of follower cost}) can be simplified if the exact form of the equivalent risk measure $\chi_{\mu}$ can be derived.
In the following, we will leverage the law-invariance property of risk measures to obtain a reformulation of (\ref{eq:follower problem}), which allows a convenient way of treating the convex combination of risk measures.

\subsection{The Kusuoka Representation}
\label{sec:reformulation:kusuoka}

\paragraph{Law-invariance} 
Consider $Z_1, Z_2\in \mathcal{Z}$.
We say that $Z_1$ and $Z_2$ have the same distribution with respect to the reference probability measure $P$ if $P(Z_1<z)=P(Z_2<z)$ for all $z\in \mathbb{R}$. 
\begin{definition}
\label{def:law-invariance}
(Law-invariant risk measure.)
A risk measure $\rho$ is law-invariant with respect to the reference probability measure $P$ if for all $Z_1, Z_2\in\mathcal{Z}$, $Z_1$ and $Z_2$ having the same distribution with respect to $P$ indicates that $\rho(Z_1)=\rho(Z_2)$.
\end{definition}
One of the most essential consequences of the law-invariance property of risk measures is the Kusuoka representation \cite{kusuoka2001law}.
Namely, we have the following equivalence for a law-invariant coherent risk measure $\rho$:
\begin{equation}
    \rho(Z)=\sup_{\sigma\in\Sigma}\int_0^1 \text{AV@R}_\alpha(Z) d\sigma(\alpha),
    \label{eq:kusuoka representation}
\end{equation}
where $\Sigma$ can be assumed to be a closed convex set of probability measures on $[0,1)$.
To utilize (\ref{eq:kusuoka representation}), hereinafter we assume that for all $\theta\in\Theta$, the risk measure $\rho_\theta$ is law-invariant and coherent.
Since we have assumed that $v^*_\theta\in \mathfrak{m}_\theta$ attains the maximum of (\ref{eq:dual representation of individual risk}), from Theorem 2 of \cite{shapiro2013kusuoka} there is a uniquely defined $\sigma_\theta$ that attains the maximum of the Kusuoka representation of $\rho_\theta$, \ie,
\begin{equation}
    \rho_\theta(Z)=\int_0^1 \text{AV@R}_\alpha(Z)d\sigma_\theta(\alpha).
    \label{eq:kusuoka maximum attained}
\end{equation}

\paragraph{Equivalent risk spectrum} 
The representation (\ref{eq:kusuoka maximum attained}) is closely related to the spectral risk measures introduced in \cite{acerbi2002spectral}.
Recall that average value-at-risk and value-at-risk has the following relation:
\begin{equation*}
    \text{AV@R}_a(Z):=(1-a)^{-1}\int_{a}^{1}\text{V@R}_t(Z)dt, a\in[0,1).
\end{equation*}
By defining $\eta_\theta(\tau)=\int_0^\tau (1-\alpha)^{-1}d\sigma_\theta(\alpha)$, $\tau\in[0,1]$, we obtain the following spectral representation of (\ref{eq:kusuoka maximum attained}):
\begin{equation}
    \rho_\theta(Z)=\int_0^1 \eta_\theta(\tau) \text{V@R}_\tau(Z)d\tau.
    \label{eq:spectral representation}
\end{equation}
The function $\eta_\theta(\cdot)$ is referred to as the risk spectrum, which is monotonically nondecreasing, right continuous, and satisfies $\int_0^1\eta_\theta(\tau)d\tau=1$ \cite{acerbi2002spectral,shapiro2013kusuoka}.

Combining (\ref{eq:spectral representation}) with (\ref{eq:follower problem}), we can rewrite the follower's cost function into
\begin{equation*}
    U_\mu(x) = \int_\Theta \int_0^1 \eta_\theta(\tau)\text{V@R}_\tau[f(x,\xi)]d\tau d\mu(\theta).
\end{equation*}
Using Fubini's theorem, we can reformulate the follower's cost as
\begin{equation}
    U_\mu(x) = \int_0^1 \eta^L_\mu(\tau) \text{V@R}_\tau[f(x,\xi)] d\tau ,
    \label{eq:equivalent spectra of the leader}
\end{equation}
where $\eta^L_\mu(\tau):=\int_\Theta \eta_\theta(\tau)d\mu(\theta)$ denotes the equivalent risk spectrum associated with leader's choice $\mu$.
Let $\sigma^L_\mu(\cdot)$ denote the probability measure attaining the maximum of the Kusuoka representation of $U_\mu(x)$. 
Then, (\ref{eq:equivalent spectra of the leader}) can be equivalently expressed as
\begin{equation}
    U_\mu(x)=\int_0^1 \text{AV@R}_\tau[f(x,\xi)] d \sigma^L_\mu(\tau).
    \label{eq:kusuoka representation of U}
\end{equation}
With a slight abuse of notation, we let $U(\sigma^L_\mu,x)= U_\mu(x)$ to emphasize the dependence of $U_\mu$ on the probability measure $\sigma^L_\mu$.

\subsection{Reformulation of the Follower's Problem}
\label{sec:reformulation:follower}
 
Observing that the Kusuoka representation (\ref{eq:kusuoka maximum attained}) naturally leads to the optimization of AV@R in the follower's problem (\ref{eq:follower problem}), we will leverage the following well-known result from \cite{rockafellar2000optimization}.
The average value-at-risk is equivalent to
\begin{equation}
    \text{AV@R}_\alpha(Z)=\min_{t\in\mathbb{R}}\left\{t+\frac{1}{1-\alpha}\mathbb{E}\left[(Z-t)_+ \right] \right\}.
    \label{eq:cvar optimization problem}
\end{equation}

\begin{proposition}
\label{prop:reformulation of follower problem}
The follower's problem (\ref{eq:follower problem}) can be reformulated into the following optimization problem:
\begin{equation}
\begin{aligned}
    \min_{x\in X} & \  \ U_\mu(x) \\
    &=\min_{x, \{t\in[0,1]\rightarrow \mathbb{R}\}}\int_0^1 (1-\alpha)t+\mathbb{E}\left[(f(x,\xi)-t)_+ \right] d\eta^L_\mu(\alpha)+\eta^L_\mu(0)\cdot \mathbb{E}[f(x,\xi)],
    \label{eq:reformulation of follower problem}
\end{aligned}
\end{equation}
which is a convex optimization problem when $f(\cdot,\xi)$ is convex for all $\xi\in\Xi$.
\end{proposition}
\begin{proof}
Substituting (\ref{eq:cvar optimization problem}) into $U_\mu(x)$, we obtain 
\begin{equation}
\begin{aligned}
    U_\mu(x) & =\int_0^1 \min_{t\in\mathbb{R}}\left\{t+\frac{1}{1-\alpha}\mathbb{E}\left[(Z_x-t)_+ \right] \right\} d\sigma_\mu(\alpha) \\
    &= \min_{t\in\{[0,1]\rightarrow \mathbb{R}\}}\int_0^1 t+\frac{1}{1-\alpha}\mathbb{E}\left[(Z_x-t)_+ \right] d\sigma_\mu(\alpha),
\end{aligned}
\end{equation}
where the second equality follows from Theorem 14.60 of \cite{rockafellar2009variational}.
Since $\eta_\theta(\tau)=\int_0^\tau (1-\alpha)^{-1}d\sigma_\theta(\alpha)$, we obtain $d\sigma_\theta(\alpha)=(1-\alpha)d\eta_\theta(\alpha)$.
Then, with the mass $\eta^L_\mu(0)$ of $\eta^L_\mu$ at $0$ specified, we arrive at the expression (\ref{eq:reformulation of follower problem}).
The convexity of the optimization problem (\ref{eq:reformulation of follower problem}) in the variable $t$ follows directly from the fact that (\ref{eq:cvar optimization problem}) is convex in $t$ \cite{rockafellar2000optimization}.
The convexity in the variable $x$ follows from the fact that $Z_x\rightarrow (1-\alpha)t+(Z_x-t)_+$ is convex and increasing for all $t$ and that $\eta^L_\mu$ is an increasing function.
\qed
\end{proof}

We have reduced the original risk sensitive stochastic programming problem in (\ref{eq:follower problem}) to a convex optimization problem as presented in Proposition \ref{prop:reformulation of follower problem}.
The remaining challenge before making the follower's problem tractable is that the decision variable $t$ in (\ref{eq:reformulation of follower problem}) is infinite dimensional.
To address this issue, we leverage the approximation technique used in \cite{acerbi2002portfolio,guo2021robust,shapiro2013kusuoka} to further simplify the follower's problem.
Recall that the equivalent risk spectrum $\eta^L_\mu$ is a monotonically nondecreasing and right continuous density function.
Consider the $n$-step function 
\begin{equation}
    \eta^{L,n}_\mu=\sum_{i=1}^n a_i\1_{[\tau_i,1]},
    \label{eq:step function}
\end{equation}
where $a_i>0$ and $0\leq \tau_1 < \tau_2 < \cdots <\tau_n<1$.
Clearly, (\ref{eq:step function}) is a step function which approximates a given equivalent risk spectrum when the parameters $a_i, i=1,2,...,n,$ are properly chosen.
The step function (\ref{eq:step function}) allows us to express the first term on the right-hand-side of (\ref{eq:reformulation of follower problem}) as:
\begin{equation*}
    \int_0^1(1-\alpha)t+\mathbb{E}[(Z_x-t)_+] d \eta^L_\mu(\alpha)
    = \sum_{i=1}^n \left\{ (1-\tau_i)t_i+\mathbb{E}[(Z_x-t_i)_+] \right\}
    \cdot (\sum_{j=1}^i a_j),
    \label{eq:refomulation of the follower problem step function}
\end{equation*}
which is linear in $a_i$ and contains finitely many variables $t_i, i=1,2,...,n$.
Accordingly, the follower's problem (\ref{eq:reformulation of follower problem}) using the approximation (\ref{eq:step function}) can be formulated as:
\begin{equation}
\begin{aligned}
     \min_{x\in X, t_1,...,t_n}\ \ 
     &U_\mu^n(x):=  \sum_{i=1}^n \left\{ (1-\tau_i)t_i+\mathbb{E}[(f(x,\xi)-t_i)_+] \right\}
    \cdot (\sum_{j=1}^i a_j) + \eta^L_\mu(0)\mathbb{E}[f(x,\xi)]\\
    \text{s.t.}\ \
    & \sum_{i=1}^n a_i\cdot \1_{[\tau_i,1]}=\sum_{m=1}^M \eta_{\theta_m}(\alpha)\cdot \mu_{\theta_m}.
    \label{eq:refomulation of the follower problem step function}
\end{aligned}
\end{equation}

The step function (\ref{eq:step function}) introduces errors to the follower's problem.
However, as long as the distance between the step function and the original equivalent risk spectrum is not too large, the solution of the follower's problem will not significantly deviate from the true solution.
Since we have the relation $d\sigma^L_\mu(\alpha)=(1-\alpha) d\eta^L_\mu(\alpha)$, the step function (\ref{eq:step function}) leads to the probability measure
\begin{equation}
    \sigma^{L,n}_\mu(\alpha)=\sum_{i=1}^n a_i(1-\tau_i)\1_{[\tau_i,1]}(\alpha).
\end{equation}
Let $\mathfrak{F}$ denote the set of functions defined by $\mathfrak{F}:=\{ \mathcal{G}:\mathcal{G}(\tau)=\text{AV@R}_\tau[f(x,\xi)], \tau\in[0,1], \forall x\in X, \forall \xi\in\Xi \}$.
Then, we can define the pseudo-metric of probability measures $\sigma_1$ and $\sigma_2$ on $[0,1]$ by:
\begin{equation}
    \dd_{\mathfrak{F}}(\sigma_1, \sigma_2):=\sup_{\mathfrak{f}\in\mathfrak{F}}
    \left|\int_0^1\mathfrak{f}d \sigma_1(\tau)-\int_0^1\mathfrak{f}d \sigma_2(\tau)\right|.
    \label{eq:pseudo metric}
\end{equation}
The next result bounds the optimal value and solution set of the follower's problem when the equivalent risk spectrum is approximated using the step function.
\begin{corollary}
\label{coro:bounds on step approximation}
The following two assertions hold.\\
(i) Let $U^*(\sigma)$ denote the optimal value of follower's problem given equivalent risk spectrum $\sigma$, \ie, 
\begin{equation*}
    U^*(\sigma)=\min_{x\in X} U(\sigma,x) =\min_{x\in X} 
    \int_0^1 \text{AV@R}_\tau [f(x,\xi)] d \sigma(\tau).
\end{equation*}
Then, 
\begin{equation}
    |U^*(\sigma^L_\mu)-U^*(\sigma^{L,n}_\mu)|\leq \dd_{\mathfrak{F}}(\sigma^L_\mu, \sigma^{L,n}_\mu).
\end{equation}
\\
(ii) Let $X^*(\sigma)$ denote the optimal set of the follower's problem given equivalent risk spectrum $\sigma$, \ie,
\begin{equation*}
    X^*(\sigma)=\argmin_{x\in X} 
    \int_0^1 \text{AV@R}_\tau [f(x,\xi)] d \sigma(\tau).
\end{equation*}
Then, there exists an open set $\mathcal{O}\supseteq X^*(\sigma^L_\mu)$ such that $X^*(\sigma^{L,n}_\mu)\subseteq \mathcal{O}$ if $d_\mathfrak{F}(\sigma^L_\mu, \sigma^{L,n}_\mu)$ is sufficiently small.
\end{corollary}
\begin{proof}
The results follow from Theorem 5 of \cite{romisch2003stability}.
\qed
\end{proof}

\subsection{Single-Level Reformulation with Finite Type Space}
\label{sec:reformulation:finite}

In this subsection, we investigate the single-level reformulation of (\ref{eq:Stackelberg problem}) for general finite RPT spaces as illustrated in Example \ref{eg:1}.
The single-level reformulation of (\ref{eq:Stackelberg problem}) for parameterized types spaces presented in Example \ref{eg:2} depends on the base risk measure of interest. 
Hence, we will elaborate one example in Section \ref{sec:reformulation:parameterized}.
Finite RPT spaces lead to $\Theta=\{\theta_1,\theta_2,...,\theta_M\}$ and $\mu=(\mu_1,\mu_2,...,\mu_M)\in \Delta^{M-1}$, where $\Delta^{M-1}$ denotes the $(M-1)$-probability simplex.

Two popular approaches to reformulate bilevel programming problems are often adopted in the literature, the Karush-Kuhn-Tucker (KKT) reformulation and the optimal value reformulation.
In the KKT reformulation, one represents the follower's problem using its KKT conditions and optimizes the leader's objective function subject to this KKT system.
It may seem that the follower's problem (\ref{eq:reformulation of follower problem}) is an unconstrained convex program whose KKT system reduces to the first-order condition of the objective function.
However, the term $\mathbb{E}[(f(x,\xi)-t_i)_+]$ in (\ref{eq:reformulation of follower problem}) is commonly treated by introducing auxiliary variables and adding corresponding constraints.
This substitution makes the follower's problem constrained and results in a KKT system containing the complementarity conditions.
The KKT reformulation based on such a KKT system turns out to be a mathematical program with equilibrium constraints (MPEC), which, according to \cite{ye1997exact}, dose not have the MFCQ. 
Therefore, we resort to the optimal value function which relies on the sensitivity of the optimal value function of the follower's problem.
For notational simplicity, we focus on the optimistic STRIPE in Remark \ref{remark:optimistic counterpart}. The reformulation of the pessimistic STRIPE only differs in the upper level.

\paragraph{Optimal Value Reformulation of STRIPE.}
The optimal value reformulation of STRIPE can be expressed as:
\begin{equation}
    \begin{aligned}
         \min_{\mu\in\Delta^{M-1}, x\in X, t_1,...,t_n} & J(\mu,x) \\
         \text{s.t.} \ \ 
         & U_\mu^n(x)-U_\mu^{n,*}\leq 0, 
         \label{eq:optimal value reformulation 0}
    \end{aligned}
\end{equation}
where $U_\mu^{n,*}$ denotes the optimal value function of (\ref{eq:refomulation of the follower problem step function}) given $\mu$.
The solution to problem (\ref{eq:optimal value reformulation 0}) can be obtained even when the exact value of $U_\mu^{n,*}$ is unavailable.
As long as we can obtain the sensitivity of $U_\mu^{n,*}$ with respect to the change of $\mu$, gradient-based algorithms can be utilized to solve (\ref{eq:optimal value reformulation 0}).
By Danskin's theorem, we know that the sensitivity of $U_\mu^{n,*}$ depends on the sensitivity of $U_\mu^{n}(x)$ at the optimal solutions of (\ref{eq:refomulation of the follower problem step function}).
In the following, we first introduce the sample approximation of (\ref{eq:optimal value reformulation 0}) to further simplify the problem; then, we present the sensitivity of the optimal value function.

\begin{proposition}
\label{prop:sample approximation}
Let $\xi_1,\xi_2,...,\xi_N$ denote the $i.i.d.$ samples of the randomness $\xi$ of size $N$.
The sampled optimal value reformulation of the STRIPE problem under the step approximation (\ref{eq:step function}) is:
\begin{equation}
    \begin{aligned}
         \min_{\substack{\mu\in\Delta^{M-1}, x\in X; \\ t_i, i=1,...,n; \\ s_i^k\geq 0, i=1,...,n, k=1,...,N }} & J(\mu,x) \\
         \text{s.t.} \ \ & U_\mu^{N,n}(x) -  U_\mu^{N,n,*} \leq 0, \\
         & f(x,\xi_j)-t_i-s_i^k \leq 0, \ \ \forall i=1,...,n, \forall k=1,...,N,
         \label{eq:sample single level reformulation}
    \end{aligned}
\end{equation}
where  $ U_\mu^{N,n,*}$ denotes the optimal value of the sampled approximation of (\ref{eq:refomulation of the follower problem step function}): 
\begin{equation}
\begin{aligned}
     \min_{\substack{x\in X, t_i, i=1,...,n;\\ s_i^k>0, i=1,...,n, k=1,...,N}} 
     & U_\mu^{N,n}(x):=  \sum_{i=1}^n c_i \left( (1-\tau_i)t_i+\frac{1}{N}\sum_{k=1}^{N}s_i^k \right)
          +c_1\frac{1}{N}\sum_{k=1}^{N}f(x,\xi_k) \\
     \text{s.t.}\ \ 
     & f(x,\xi_k)-t_i-s_i^k\leq 0 , \ \ \forall i=1,...,n, \forall k=1,...,N.
     \label{eq:sample follower problem}
\end{aligned}
\end{equation}
with $c_i=\sum_{m=1}^M\mu_m\eta_{\theta_m}(\tau_i) $ for all $i=1,2,...,n$.
\end{proposition}
\begin{proof}
With the introduction of auxiliary variables $s_i^k, i=1,...,n, k=1,...,N$, the sampled approximation of (\ref{eq:refomulation of the follower problem step function}) can be formulated as 
\begin{equation*}
    \begin{aligned}
        \min_{\substack{x\in X, t_i, i=1,...,n;\\ s_i^k>0, i=1,...,n, k=1,...,N}} 
        & \sum_{i=1}^n (\sum_{j=1}^i a_j)\left( (1-\tau_i)t_i+ \frac{1}{N}\sum_{k=1}^n s_i^k \right)
     + \eta^L_\mu(0)\frac{1}{N}\sum_{k=1}^N f(x,\xi_k) \\
    \text{s.t.}\ \ 
    & f(x,\xi_k)-t_i-s_i^k\leq 0,  \ \ \forall i=1,...,n, \forall k=1,...,N, \\
    &  \sum_{i=1}^n a_i \1_{[\tau_i,1]}=\sum_{m=1}^M \mu_m\eta_{\theta_m}(\alpha).
    \end{aligned}
\end{equation*}
Observe from (\ref{eq:step function}) that $\sum_{m=1}^M\eta_{\theta_m}(t_j)\mu_m=\sum_{i=1}^j a_i$, $\forall j=1,...,n$, and by introducing the variables $c_i, i=1,...,n$, we arrive at the formulation (\ref{eq:sample follower problem}).
\qed
\end{proof}

\paragraph{Sensitivity of the Optimal Value Function.}
\begin{proposition}
\label{prop:sensitivity}
Let $x^*(\mu)$, $t_i^*(\mu), \forall i=1,...,n$, and $s_i^{k,*}(\mu), \forall i=1,...,n, \forall k=1,...,N$ denote an optimal solution of (\ref{eq:sample follower problem}) given $\mu$.
Then, the sensitivity of the optimal value $U_\mu^{N,n,*}$ with respect to $m=1,2,...,M,$ is given by
\begin{equation}
    \frac{\partial}{\partial \mu_m} U_\mu^{N,n,*}
    =\sum_{1=i}^n \eta_{\theta_m}(\tau_i)\left( (1-\tau_i)t_i^*+\frac{1}{N}\sum_{k=1}^N s_i^{k,*} \right) + \frac{\eta_{\theta_m}(\tau_i)}{N}\sum_{k=1}^N f(x^*,\xi_k).
    \label{eq:sensitivity of optimal value function}
\end{equation}
\end{proposition}
\begin{proof}
Since the constraint in (\ref{eq:sample follower problem}) is independent of the parameter $\mu$, the result follows from Danskin's theorem.
\qed
\end{proof}

The issues of constraint qualifications can occur in problem (\ref{eq:sample single level reformulation}).
A suitable way to resolve it is by relaxing the first constraint in (\ref{eq:sample single level reformulation}) to $U_\mu^{N,n}(x)-U_\mu^{N,n,*}\leq \epsilon$ for $\epsilon>0$ and solving for $\epsilon$-optimal follower's solution.
In \cite{lin2014solving}, the authors have shown that MFCQ holds automatically under this relaxation technique.

Note that the optimal solution $x^*(\mu)$, $t_i^*(\mu), \forall i=1,...,n$, $s_i^{k,*}(\mu)$, $\forall i=1,...,n, \forall k=1,...,N$  required in Proposition \ref{prop:sensitivity} is obtained by solving the follower's problem given the leader's action $\mu$.
This fact suggests that the numerical computation of the solution to (\ref{eq:sample single level reformulation}) still involves a two-time scale iteration.
For advances in the computation methods of bilevel programming problems, we refer the readers to \cite{lin2014solving,chen2022single} for single-time scale algorithms.

\paragraph{Approximation Error.}
The sample approximation (\ref{eq:sample single level reformulation}) introduces errors to the solution to STRIPE.
In this subsection, we analyze the performance of the sample approximation (\ref{eq:sample single level reformulation}) centered around Definition \ref{def:extended stackelberg solution}.
We first invoke the following result \cite{shapiro2021lectures} quantifying the statistical property of the follower's problem with finitely many samples.

\begin{lemma}
\label{lemma:monte carlo finite sample}
Let $\xi_1,\xi_2,...,\xi_N$ denote $N$ $i.i.d$ samples of the randomness $\xi$.
Let the following conditions be satisfied:
 \\
\hspace*{1em}(C1) There are functions $\kappa^i_1:\Xi\rightarrow\mathbb{R}_+$ for all $i=1,2,...,n$ and $\kappa_2:\Xi\rightarrow\mathbb{R}_+$, such that their moment generating functions $\mathbb{M}_{\kappa_1^i}(\cdot)$ for all i=1,2,...,n, and $\mathbb{M}_{\kappa_2}(\cdot)$ are finite valued in a neighborhood of zero, and 
\begin{equation}
    \begin{aligned}
         &|(f(x',\xi)-t_i)_+-(f(x,\xi)-t_i)_+|\leq \kappa_1^i(\xi)||x'-x||, \forall i=1,2,...,n, \\
         &|f(x',\xi)-f(x,\xi)|\leq\kappa_2(\xi)||x'-x||,
    \end{aligned}
\end{equation}
for almost everywhere $\xi\in\Xi$ and all $x',x\in X$;\\
\hspace*{1em}(C2) The moment generating function $\mathbb{M}_{x',x}(\cdot)$ of the random variable
\small
\begin{equation*}
\begin{aligned}
    Y_{x',x}=&\left[ \left( \sum_{i=1}^M c_i(f(x',\xi)-t_i)_+ +c_1 f(x',\xi) \right)-\left( \sum_{i=1}^M c_i\mathbb{E}[(f(x',\xi)-t_i)_+] +c_1\mathbb{E}[f(x',\xi)] \right) \right] \\
    - 
    &\left[ \left( \sum_{i=1}^M c_i(f(x,\xi)-t_i)_+ +c_1 f(x,\xi) \right)-\left( \sum_{i=1}^M c_i\mathbb{E}[(f(x,\xi)-t_i)_+] +c_1\mathbb{E}[f(x,\xi)] \right) \right],
\end{aligned}
\end{equation*}
\normalsize
satisfies 
\begin{equation*}
    \mathbb{M}_{x',x}(w)\leq \exp (\lambda^2||x'-x||^2\frac{w^2}{2}).
\end{equation*}
\\
Then, for any choice of $\mu\in \Delta^{M-1}$, a follower's action $x^*_N\in X$ is an $\epsilon_1$-optimal solution to (\ref{eq:refomulation of the follower problem step function}) with probability at least $1-\beta$ for $\beta\in(0,1)$, if it is an $\epsilon_2$-optimal solution to (\ref{eq:sample follower problem}) with $\epsilon_2\in[0,\epsilon_1)$ and sample size
\begin{equation}
    N\geq\frac{O(1)\lambda^2 D^2}{(\epsilon_1-\epsilon_2)^2}\left[n\cdot ln(\frac{O(1)\mathbb{E}[\kappa(\xi)]D}{\epsilon_1-\epsilon_2})+ln(\frac{1}{\beta}) \right],
    \label{eq:sample size bound}
\end{equation}
where $D$ denotes the diameter of set $X$ and $\kappa(\xi)=\sum_{i=1}^n c_i\kappa_1^i(\xi)+c_1 \kappa_2(\xi)$.
\end{lemma}
\begin{proof}
See Corollary 5.19 of \cite{shapiro2021lectures}.
\qed
\end{proof}

With the above lemma, we conclude the following result on the performance of the solution to the sampled problem (\ref{eq:sample follower problem}) evaluated using the lower level problem in (\ref{eq:Stackelberg problem}).

\begin{theorem}
\label{thm:sample approximation}
Let $\xi_1,\xi_2,...,\xi_N$ denote $N$ $i.i.d.$ samples of the randomness $\xi$.
Assume that the conditions of Lemma \ref{lemma:monte carlo finite sample} hold. 
Then, for any $\mu\in \Delta^{M-1}$, a point $x^*_N$ solving (\ref{eq:sample follower problem}) is an $\epsilon$-optimal solution of (\ref{eq:follower problem}) for $\epsilon:= U^*(\sigma^L_\mu)+2d_{\mathfrak{F}}(\sigma^L_\mu,\sigma^{L,n}_\mu)+\epsilon_1$ if the sample size $N$ satisfies (\ref{eq:sample size bound}).
\end{theorem}
\begin{proof}
First, we observe from Lemma \ref{lemma:monte carlo finite sample} that, with probability at least $1-\beta$,  $x^*_N$ is an $\epsilon_1$-optimal solution of (\ref{eq:refomulation of the follower problem step function}), \ie, $U(\sigma^{L,n}_\mu, x^*_N)\leq U^*(\sigma^{L,n}_\mu)+\epsilon_1$.
From (ii) of Corollary \ref{coro:bounds on step approximation}, we deduce that $|U^*(\sigma^L_\mu)-U^*(\sigma^{L,n}_\mu)|\leq d_{\mathfrak{F}}(\sigma^L_\mu, \sigma^{L,n}_\mu)$.
Then, we obtain that, with probability at least $1-\beta$, the following condition holds:
\begin{equation}
    U(\sigma^{L,n}_\mu,x^*_N)\leq U^*(\sigma^L_\mu)+d_{\mathfrak{F}}(\sigma^L_\mu,\sigma^{L,n}_\mu)+\epsilon_1.
    \label{eq:proof:intermediate bound 1 on x^*_N}
\end{equation}
The representation (\ref{eq:kusuoka representation of U}) leads to 
\begin{equation}
    \begin{aligned}
         U(\sigma^L_\mu,x^*_N)
         &=\int_0^1 \text{AV@R}_\tau [f(x^*_N,\xi)] d \sigma^L_\mu \\
         &=\int_0^1 \text{AV@R}_\tau [f(x^*_N,\xi)]\left( d\sigma^L_\mu-d\sigma^{L,n}_\mu+d\sigma^{L,n}_\mu \right)\\
         &=\int_0^1 \text{AV@R}_\tau [f(x^*_N,\xi)]d\sigma^{L,n}_\mu+\int_0^1 \text{AV@R}_\tau [f(x^*_N,\xi)]\left(d\sigma^L_\mu-d\sigma^{L,n}_\mu \right).
         \label{eq:proof:intermediate bound 2 on x^*_N}
    \end{aligned}
\end{equation}
By definition of $d_{\mathfrak{F}}(\cdot,\cdot)$, (\ref{eq:proof:intermediate bound 1 on x^*_N}) and (\ref{eq:proof:intermediate bound 2 on x^*_N}) lead to the following inequality:
\begin{equation*}
    \begin{aligned}
         U(\sigma^L_\mu,x^*_N)
         &\leq U(\sigma^{L,n}_\mu,x^*_N)+d_{\mathfrak{F}}(\sigma^L_\mu,\sigma^{L,n}_\mu)\\
         &\leq U^*(\sigma^L_\mu)+2d_{\mathfrak{F}}(\sigma^L_\mu,\sigma^{L,n}_\mu)+\epsilon_1.
    \end{aligned}
\end{equation*}
Denote $\epsilon:= U^*(\sigma^L_\mu)+2d_{\mathfrak{F}}(\sigma^L_\mu,\sigma^{L,n}_\mu)+\epsilon_1$.
Then, with probability at least $1-\beta$, $x^*_N$ is an $\epsilon$-optimal solution of $\min_{x\in X}U(\sigma^L_\mu, x)$.
Finally, suppose that $\sigma^L_\mu$ is obtained by choosing the probability measure $\Tilde{\mu}\in\mathcal{Q}$, \ie, $\eta^L_\mu(\tau)=\sum_{m=1}^M\eta_{\theta_m}(\tau)\Tilde{\mu}_{\theta_m}$ and $d\sigma^L_\mu(\alpha)=(1-\alpha)d\eta^L_\mu(\alpha)$.
Then, $x^*_N$ can be considered as an $\epsilon$-optimal solution of $\min_{x\in X}U_{\Tilde{\mu}}(x)$ with probability at least $1-\beta$.
This completes the proof.
\qed
\end{proof}

The assertions of Theorem \ref{thm:sample approximation} connect the sampled optimal value reformulation of STRIPE in Proposition \ref{prop:sample approximation} with the $\epsilon$-robust $\delta$-approximate Stackelberg equilibrium of STRIPE proposed in Definition \ref{def:extended stackelberg solution}.
In particular, Theorem \ref{thm:sample approximation} shows the possibility of obtaining $\epsilon$-approximate follower's responses given arbitrary leader's choice of RPT distribution with high probability. 
Together with Theorem \ref{thm:optimality compromise}, one can also obtain the bound on the leader's ambiguity tolerance induced by the approximation of the randomness $\xi$ with $N$ samples in (\ref{eq:sample follower problem}) and the approximation with the $n$-step function (\ref{eq:step function}).

\subsection{An Example of the Single-Level  Reformulation with Parameterized Type Spaces}
\label{sec:reformulation:parameterized}
When the RPT space is the space containing parameterized risk measures as presented in Example \ref{eg:2}, the single-level reformulation of STRIPE will largely depend on the choice of the base risk measure.
Here, we use an example to show that the STRIPE problem can be significantly simplified by reformulation techniques when we adopt parameterized type spaces.

\begin{example}
Consider the scenario where individuals in the population possess risk preferences described by the absolute semideviation measures. 
Specifically, let $\rho_\theta[Z]=\mathbb{E}[Z]+\theta \mathbb{E}\{[z-\mathbb{E}(Z)]_+\}$ with $\theta\in[\theta_1,\theta_2]\subset(0,1)$.
We work in the space $\mathcal{Z}=\mathcal{L}_1(\Xi, \mathcal{F}, P)$.
Let $\kappa=\text{Pr}\{Z>\mathbb{E}[Z]\}$.
Then, we can express $\eta_\theta$ as follows \cite{shapiro2013kusuoka}:
\begin{equation*}
    \eta_\theta(\xi)=\begin{cases}
    1-\theta\kappa, \text{ if } 0\leq\xi < 1-\kappa, \\
    1+\theta(1-\kappa), \text{ if } 1-\kappa\leq
    \xi \leq1.
    \end{cases}
\end{equation*}
Since the random cost $Z$ can be chosen from the set $\mathcal{Z}$, $\kappa$ can be any number in $(0,1)$.
Thus, though the dual set of the mean semideviation measure is not generated by one element in $\mathcal{Z}^*$ \cite{shapiro2013kusuoka}, we can still express its the Kusuoka representation as follows:
\begin{equation*}
    \rho_\theta[Z]=\sup_{\kappa\in(0,1)}\left\{ (1-\theta\kappa)\text{AV@R}_0[Z] + (\theta\kappa)\text{AV@R}_{1-\kappa}[Z] \right\}.
    \label{eq:kusooka of mean semideviation}
\end{equation*}
Using the optimization formulation of the AV@R in \cite{rockafellar2000optimization}, we can express the risk evaluated by an individual with type $\theta$ as
\begin{equation}
     \rho_\theta[Z]=\sup_{\kappa\in(0,1)} \inf_{t\in\mathbb{R}} \mathbb{E} \left\{ Z+\theta\kappa(t-Z)+\theta[Z-t]_+ \right\}.
     \label{eq:maxmin of kusuoka of mean semidiviation}
\end{equation}
Assuming that the distribution function $\phi_Z(\cdot)$ of $Z$ is continuous, we obtain that the unique saddle point of (\ref{eq:maxmin of kusuoka of mean semidiviation}) is $(\phi_Z(\mathbb{E}[Z]), \mathbb{E}[Z])$.
Since the equivalent risk spectra of the leader is $\eta^L_\mu(\tau):=\int_\Theta \eta_\theta(\tau)d\mu(\theta)$, $\eta^L_\mu(\xi)$ is also a piecewise constant function having one point of discontinuity at $\xi=1-\kappa$ as follows:
\begin{equation*}
    \eta^L_\mu(\xi)=\begin{cases}
    1-\kappa\mathbb{E}_\mu[\theta], \text{ if } 0\leq\xi < 1-\kappa, \\
    1+(1-\kappa)\mathbb{E}_\mu[\theta], \text{ if } 1-\kappa\leq
    \xi \leq1.
    \end{cases}
\end{equation*}
Since $\mu\in\mathcal{Q}$, $\mathbb{E}_\mu[\theta]$ can admit any value from the interval $[\theta_1,\theta_2]$ when the choice of $\mu$ is proper.
Therefore, to optimize $L(x)$ in (\ref{eq:leader loss func}) is equivalent for the leader to directly choose a value of $\mathbb{E}_\mu[\theta]$ instead of choosing the RPT distribution $\mu$.
Possible RPT distributions that lead to the optimal $L(x)$ and minimize the distance term $W_1(\mu,\mu^0)$ can then be computed by comparing the equivalent risk spectra $\eta^L_{\mu}$ given $\mathbb{E}_\mu[\theta]$ with the risk spectra associated with the risk measures $\rho_\theta$.
Finally, the solution to the STRIPE problem can be found by comparing the total costs of the leader with different choices of the value of $\mathbb{E}_\mu[\theta]$.
Similar techniques can be applied to scenarios where other classes of parametric risk measures are chosen.
\end{example}

\section{Applications}
\label{sec:case}

In this section, we discuss two distinct applications that are unified within the STRIPE framework.
In the first application, we use STRIPE to investigate contract problems with risk preference design. 
We demonstrate the application of the approximate Stackelberg equilibrium in Definition \ref{def:extended stackelberg solution} in contract design problems and elaborate on its implication in incentive compatibility and equilibrium selection.
In the second application, we show the role of STRIPE in meta-learning problems.
In particular, we leverage the sensitivity analysis introduced in Section \ref{sec:reformulation:finite} to address the challenge of how to adapt the meta-models.

\subsection{Contract Problem with Risk Preference Design}
\label{sec:case:contract}

We consider the classical P-A moral hazard problem \cite{stole2001lectures}, which is inherently a bilevel optimization problem when the agent's action is hidden in the eye of the principal. The P-A problem has been used for contract and mechanism designs in economic sciences and has also been widely applied in other domains such as computer science and engineering.

We consider a contract design problem with one principal and one agent.
Following the setting of Section \ref{sec:problem}, the agent in the contract problem is an idiosyncratic individual from a population with RPTs $\theta\in \Theta:=\{\theta_1,\theta_2,...,\theta_M\}$ who chooses action $x\in X$.
The uncertainty $\xi\in \Xi$ denotes  the principal's loss which is affected by the agent's action.
An example of the set $\Xi$ is the interval $\Xi=[\underline{\xi},\Bar{\xi}]$, where $\underline{\xi}$ represents the minimum loss to the principal and $\Bar{\xi}$ denotes the loss in the worst-case scenario.
We consider $P(\xi, x)$ as the reference probability measure parameterized by the agent's action.

The principal determines a contract function $w:\Xi \rightarrow \mathbb{R}$ to minimize her own objective function.
An example of a contract is the wage transferred from an employer to an employee.
The risk-neutral principal's problem is captured by
\begin{equation*}
    \int_{\Xi}w(\xi) + \xi dP(\xi,x).
    \label{eq:principal problem}
\end{equation*}
The objective function of the agent is
\begin{equation*}
   \sum_{i=1}^n\rho_{\theta_i}[U(w(\xi), \xi ,x)]\cdot \mu_i,
\end{equation*}
where $U:\mathbb{R}\times \Xi \times X\rightarrow \mathbb{R}$ denotes the loss function of the agent and $\mu\in\mathcal{Q}$ is the distribution of the RPTs.
The hidden-action P-A problem with risk-preference design is described as follows:
\begin{equation}
    \begin{aligned}
         \min_{w(\cdot), \mu(\cdot), x}
         & \int_{\Xi}w(\xi) + \xi dP(\xi,x)+\gamma W_1(\mu, \mu^0) \\
         \text{s.t. } \ \ 
         & \sum_{i=1}^n\rho_{\theta_i}[U(w(\xi), \xi ,x)]\cdot \mu_i\leq \Bar{U}, \text{(IR)}, \\
         & x\in \argmin_{x'\in X}  \sum_{i=1}^n\rho_{\theta_i}[U(w(\xi), \xi ,x')]\cdot \mu_i, \text{(IC)}, \\
         \label{eq:contract problem}
    \end{aligned}
\end{equation}
where IR represents the individual rationality constraint guaranteeing beneficial participation and IC refers to the incentive compatibility constraint indicating that the agent optimizes her own objective.
Note that the minimization over $x$ in (\ref{eq:contract problem}) indicates that the principal can control the feasible actions of the agent and that the problem is formulated in an optimistic perspective.

The formulation (\ref{eq:contract problem}) enriches the standard contract problem in two dimensions.
First, the risk measure captures not only the nonlinearity of the utility attitude of the agent but also her preference towards probabilistic risks.
Second, the RPT distribution equips the principal with an additional design parameter, which provides an extra degree of freedom to react to the scenario of hidden actions.

The IC constraint in (\ref{eq:contract problem}) is based on the fundamental assumption that the agents are prefect optimizers.
It indicates that, given a contract, the agent can obtain the exact action that optimizes her own objective function.
Sometimes agents can be, however, limited by their observations or computational capacity to act in a suboptimal way. 
It is necessary to consider approximate IC in P-A problems.

The $\epsilon$-approximate IC in the context of (\ref{eq:contract problem}) can be formulated as:
\begin{equation}
    \sum_{i=1}^n\rho_{\theta_i}[U(w(\xi), \xi ,x)]\cdot \mu_i
     \leq 
     \min_{x'\in X}  \sum_{i=1}^n\rho_{\theta_i}[U(w(\xi), \xi ,x')]\cdot \mu_i + \epsilon.
     \label{eq:app IC}
\end{equation}
With (\ref{eq:app IC}) in lieu of the IC constraint of (\ref{eq:contract problem}), the agents in the P-A problem cannot play their optimal strategies.
There is a need for choosing a perspective toward the set of feasible agent's actions, that is, either pessimism or optimism.
The choice of perspectives affects the principal's revenue of the contract.
If we choose, in accordance with Definition \ref{def:extended stackelberg solution}, the pessimistic perspective, the principal has no control over the actions the agent can choose, and she adopts the robust solution, \ie, she must choose the optimal contract and the RPT distribution under the worst-case scenario agent's action satisfying (\ref{eq:app IC}).
Then, Theorem \ref{thm:optimality compromise} implies that pessimism decreases the principal's revenue at the worst possible rate of $O(\epsilon^{1/2})$.

On the contrary, the optimistic perspective provides the principal with the ability of forcing the agents to choose certain actions; \ie, the principal can select equilibrium solutions to mechanism design problems.
In this scenario, it is straightforward to observe that the $\epsilon$-approximate IC benefits the principal, since the relaxation of the IC constraint enlarges the feasible set of the principal's decision variable.
In fact, a classical upper bound for the principal's revenue benefit when the IC is relaxed by $\epsilon$ is in the order of $O(\epsilon^{1/2})$ (see, for example, \cite{balseiro2022mechanism} and the references therein).
The rate coincides with what we have discovered in Theorem \ref{thm:optimality compromise} under the pessimistic perspective.

Note that the formulation of (\ref{eq:contract problem}) has been investigated in \cite{liu2022mitigating} in the context of cyber insurance.
It has focused on the characterization of the optimal contract plan.
In particular, using the proposed metric that measures the intensity of the moral hazard issue, we have shown that risk preference design can mitigate moral hazard.
This fact indicates the reduction of the deviation from the second-best contract obtained from solving the hidden-action problem to the first-best contract obtained from solving the full-information benchmark problem.
The risk preference design has implications in the following two scenarios to improve the performance of the contract design.
In the first scenario, the principal solves the hidden-action contract problem with risk preference design. The principal obtains an optimal contract whose performance is close to the one obtained from the full-information benchmark.
In the second scenario, the principal solves the 
full-information counterpart of (\ref{eq:contract problem}). 
The risk preference design can mitigate moral hazard and help achieve the performance attained by the full-information case.

\subsection{Guided Risk-Sensitive Meta-Learning}
\label{sec:case:learning}

Meta-learning \cite{vanschoren2018meta} is an emerging branch of machine learning that focuses on finding an averaged model or a representative parameter vector for fast adaptations to a class of related learning tasks.
The STRIPE framework is closely related to meta-learning.
In particular, we show that the follower's problem in our framework resembles a risk-sensitive version of the meta-learning problem of \cite{finn2017model}. 
This resemblance enriches the learning tasks with risk attitudes toward model uncertainties or data perturbations. 
The leader's problem in our framework plays the role of providing guidance to the selection of the meta-parameter.
This guidance acts as an exogenous criterion on the meta-parameter, such as an autonomous driver's license in autonomous driving problems.

We take the binary classification problem as an example of statistical learning problems.
Consider the random vector $\xi=(\mathfrak{x}, \mathfrak{y})$ with $\mathfrak{x}\in\mathbb{R}^n$ denoting the input data and $\mathfrak{y}\in\{+1, -1\}$ denoting the output label. 
For classification problems, the model $f(x, \xi)$ can be formulated as $\mathfrak{L}(\mathfrak{y}\cdot \mathfrak{x}^T x)$ with $\mathfrak{L}(z)$ representing the loss function, such as a hinge loss $\max\{0, 1-z\}$ in support vector machines or a log-loss $log(1+e^{-z})$ in logistic regression problems.
Then, in the follower's problem (\ref{eq:follower problem}), $\rho_\theta[f(x,\xi)]$ denotes a
risk-sensitive learning problem, which is closely related to the distributionally robust learning problem \cite{kuhn2019wasserstein} as depicted by the equivalent formulation (\ref{eq:dual representation of individual risk}).
A type $\theta$ in this scenario represents the adversarial behaviors due to the errors in the data or malicious manipulations of the dataset. The distributional robustness aims to find the learning parameters which is suitable under the worst-case data distribution.
With the averaging over a distribution $\mu$ of tasks with robustness consideration $\theta$, $\mathbb{E}_{\theta \sim \mu}[\rho_\theta[f(x,\xi)]]$ can be interpreted as a meta-objective (see \cite{finn2017model}).
If we consider the one-shot gradient update with step-size $\mathfrak{a}$, the follower's problem becomes:
\begin{equation}
    \min_{x\in X} \mathbb{E}_{\theta\sim\mu} 
    \left[ \rho_\theta [ \mathfrak{L}(\mathfrak{y}\cdot \mathfrak{x}^T(x-\mathfrak{a}\nabla_x\rho_\theta[\mathfrak{L}(\mathfrak{y}\cdot x^T\mathfrak{x})]))] \right].
    \label{eq:meta learning}
\end{equation}
The formulation in (\ref{eq:meta learning}) takes the form of a risk sensitive meta-learning problem.
It has been shown in \cite{finn2017model} that the solution to a meta-learning problem can adapt to individual tasks even with only one gradient update.

In our framework, we have a leader's problem guiding the learning of the meta-parameter.
In particular, the loss $J_0(x)$ has the interpretation of a test or a criterion for the meta-classifier. 
The distribution $\mu$ in the term $W_1(\mu, \mu_0)$ can be considered as a manipulation of robust considerations such that the meta-classifier learned can better meet the criterion $J_0(x)$.
The result in Section \ref{sec:asymmetry} serves as a useful method to identify a given learned model,  such as a classifier or a regressor, by assigning it a distribution of learning tasks.

Adaptation of meta-parameters to specific learning tasks is one essential focus of meta-learning. 
From the formulation of the meta-objective (\ref{eq:meta learning}), we observe that the one-shot gradient update $x-\mathfrak{a}\nabla_x\rho_\theta[\mathfrak{L}(\mathfrak{y}\cdot \mathfrak{x}^T x)]$ has taken into account the adaptation of the meta-parameters after they are solved using (\ref{eq:meta learning}).
Next, we show that using the STRIPE framework, we can obtain performance estimates of the adaptation of meta-parameters.

For a given learning task distribution $\mu\in\mathcal{Q}$, let $x^*_{\mu,\mathfrak{L}}$ and $U^{N,n,*}_{\mu,\mathfrak{L}}$ denote the optimal solution to and the optimal value of problem (\ref{eq:meta learning}) using the $n$-step approximation and the $N$-sample approximation as presented in Proposition \ref{prop:sample approximation}.
Let $\mathrm{1}_m$ denote the Dirac measure such that $\mathrm{1}_m(\theta_m)=1$ and $\mathrm{1}_m(\theta)=0$ for $\theta\neq \theta_m$.
We use $U^{N,n,*}_{\mathrm{1}_m,\mathfrak{L}}$ to denote the value of the meta-objective function in (\ref{eq:meta learning}) given task distribution $\mathrm{1}_m$ evaluated at $x^*_{\mu,\mathfrak{L}}$.
The adaptation of the meta-parameter  $x^*_{\mu,\mathfrak{L}}$ to the learning task $\theta_m$ is $x^*_m=x^*_{\mu,\mathfrak{L}}-\mathfrak{a}\nabla_x\rho_\theta[\mathfrak{L}(\mathfrak{y}\cdot \mathfrak{x}^T x^*_{\mu,\mathfrak{L}})]$.
The performance of this adaptation is then computed using $\rho_{\theta_m}[\mathfrak{L}(\mathfrak{y}\cdot \mathfrak{x}^T x^*_m)]$, which is equivalent to  $U^{N,n,*}_{\mathrm{1}_m,\mathfrak{L}}$.
Therefore, using the sensitivity result from Proposition \ref{prop:sensitivity}, we can obtain the adaptation performance estimate.
For example, the linear approximation of $U^{N,n,*}_{\mathrm{1}_m,\mathfrak{L}}$ using $U^{N,n,*}_{\mu,\mathfrak{L}}$  can be expressed as
\begin{equation}      \Tilde{U}^{N,n,*}_{\mathrm{1}_m,\mathfrak{L}}=U^{N,n,*}_{\mu,\mathfrak{L}}+\left(\nabla_{\mu}U^{N,n,*}_{\mu,\mathfrak{L}}\right)^T \cdot (\mathrm{1}_m-\mu),
    \label{eq:adaptation performance estimate}
\end{equation}
where $\nabla_{\mu}U^{N,n,*}_{\mu,\mathfrak{L}}$ denotes the gradient vector of the optimal value $U^{N,n,*}_{\mu,\mathfrak{L}}$ computed using (\ref{eq:sensitivity of optimal value function}).

Since the value in (\ref{eq:adaptation performance estimate}) can be computed for all tasks within the class of learning tasks of interests, we can obtain the complete adaptation performance estimate of the optimal meta-parameter. 
These performance estimates can be further utilized to improve the meta-parameter.
For example, if we aim to obtain a meta-parameter such that the adaptation performance in the worst-case scenario is guaranteed, then we can add the following constraint to the lower level:
\begin{equation*}
    \max_{m=1,2,...,M} \Tilde{U}^{N,n,*}_{\mathrm{1}_m,\mathfrak{L}} \leq \hat{U},
\end{equation*}
where $\hat{U}$ is a predetermined performance level.

The performance estimates (\ref{eq:adaptation performance estimate}) can reduce the computation complexity of (\ref{eq:meta learning}).
The expectation over $\theta$ with respect to the distribution $\mu$ in (\ref{eq:meta learning}) indicates that gradient terms associated with all the learning tasks are computed to update the decision variable $x$.
However, as we gradually approach the optimal meta-parameter during the iterations of an optimization algorithm, the meta-parameter becomes well-adaptable to certain learning tasks, as measured by the performance estimates.
In the adaptation stage, we can only update the meta-parameter using the gradient terms associated with the learning tasks that the current meta-parameter cannot be effectively adapted to, and consequentially reduce the computation complexity in the lower level.

\section{Concluding Remarks}
\label{sec:conclusion}
In this paper, we have proposed the Stackelberg risk preference design framework to optimally shape the distribution of risk attitudes in a population of decision-makers.
Considering an idiosyncratic individual from the population who possesses the average risk preference, we have represented the population using a representative and formulated a Stackelberg game between a designer and an agent.
We have defined the $\epsilon$-Robust $\delta$-Approximate Stackelberg Equilibrium to study the $\delta$-approximate design of the leader when
the follower's response is $\epsilon$-optimal.
Based on the growth conditions in the stability analysis of optimization problems, we have derived an upper bound on the leader's design cost using the primitive perception gap which measures the distance between the uncontrolled risk preference distribution and the desired one.
The consideration of $\epsilon$-robustness in the follower's problem introduces a cost for the leader that is upper bounded in the order of $\sqrt{\epsilon}$.
This result not only aids the parameter selection in the design problem, but also provides insights on a class of Principal-Agent problems with approximate incentive compatibility constraints.
We have reformulated the Stackelberg risk preference design problem 
into a single-level optimization problem leveraging the spectral representations of risk measures.
Furthermore, a data-driven approach has enabled the computation of 
$\epsilon$-robust follower's solutions with high probability given a sufficient number of samples.
This approach has a strong implication for risk-sensitive meta-learning problems in the literature in that it provides a way to estimate the adaptation performances of the meta-parameter to reduce the computation complexity.

Future works would include the following directions.
One could enrich the STRIPE framework by considering multiple followers and their interdependencies. 
It would be useful to extend the follower’s risk-sensitive decision-making problem into its multistage counterpart and study the design of the trajectory of risk-preference type distributions in dynamic environments. 
Furthermore, the leader could choose a perspective other than optimism or pessimism toward the set of the follower’s solutions leveraging distributional information of this set.
It would lead to risk-sensitive leader's decisions.

\section*{Conflicts of Interests}
The authors declare no conflict of interests.

\bibliographystyle{spmpsci}      
\bibliography{bibliography.bib}
\nocite{*}

\end{document}